\documentclass{article}
\usepackage{arXiv}

\usepackage{fancyhdr}
\title{
    \textbf{Mitigating Numerical Stiffness in Least-Squares Formulations of Elliptic PDEs for Physics-Informed Neural Networks 
    }
}
\author{ 
    Phil-Alexander Hofmann\thanks{
        Center for Advanced Systems Understanding, 
        Helmholtz-Zentrum Dresden-Rossendorf e.V.,
        Untermarkt 20, 02826 Görlitz, 
        Germany, 
        \texttt{p.hofmann@hzdr.de} }, \
    Michael Hecht\thanks{ 
        Mathematical Institute, 
        University of Wrocław, 
        pl. Grunwaldzki 2/4, 
        50-384 Wrocław, 
        Poland, 
        \texttt{michael.hecht@math.uni.wroc.pl} 
    } 
}

\usepackage{lipsum}

\begin{document}

\maketitle
\begin{abstract}
We present theoretical insights into \(H^{-1}\) residual loss formulations of physics-informed neural networks (PINNs) for learning solutions of partial differential equations (PDEs). 
Standard PINN formulations use a multi-term loss functional consisting of interior and boundary loss terms that are based on \(L^2\)-residuals and discretized as mean square errors (MSE). 
Imbalanced magnitudes of these terms cause numerical stiffness phenomena, resulting in ill-conditioning and slow convergence. 
In this work, we analyze discretizations of the \(H^{-1}\)-norm that are used in the context of elliptic PDEs with arbitrary, nonzero Dirichlet boundary conditions. 
We prove that these \(H^{-1}\) discretizations rebalance the PDE loss, improve conditioning, and mitigate stiffness effects compared with the standard MSE discretization.
We validate our theoretic results through operator-level experiments with randomly sampled residuals and PINN experiments for the Poisson and stationary incompressible Navier-Stokes equations. These experiments confirm the numerical effectiveness of the proposed rebalancing for elliptic PDEs and, more broadly, for problems with elliptic behavior.
\end{abstract}


\section{Introduction}
\label{sec:introduction}

Across disciplines, partial differential equations (PDEs) are ubiquitous, serving as the primary framework for expressing physical laws and dynamics in systems such as fluid flow, heat transfer, and electromagnetism \cite{Jost2007, Brezis2011}. 
Although, classic numerical methods for PDEs have been developed into a rich and rigorous theory, including finite element methods \cite{Ern2004}, finite difference methods \cite{LeVeque2007}, finite volume methods \cite{Eymard2000}, spectral methods \cite{Bernardi1997, Canuto2007}, and meshfree methods \cite{Li2004}.   However, their applicability is typically tied to inherent choices of meshes, basis functions, and boundary constraints. 
This inflexibility has long been recognized and addressed within \emph{Rayleigh–Ritz–Galerkin methods} \cite{Bramble1970, Babuska1973, Barrett1986}. Here, trial spaces that do not need to satisfy boundary conditions a priori and rely on boundary constraints enforced weakly or, alternatively, by boundary penalty terms.

Since 2017, physics–informed neural networks (PINNs) have emerged as a framework for solving ordinary and partial differential equations through residual-based losses \cite{Raissi2017a, Raissi2017b, Raissi2019}. Despite their empirical success across a wide range of applications \cite{Lawal2022, Cai2021, Toscano2025}, it is well-known that PINNs suffer from severe stiffness and training pathologies \cite{Wang2021, Maddu2022, Hu2024}.

Motivated by the need to mitigate stiffness caused by boundary penalties and gradient imbalance, the \(H^{-1}\) residual loss formulation exploits the connection between PINNs and Rayleigh-Ritz-Galerkin methods \cite{Suarez2024}. 
The existing analysis in \cite{Suarez2024}, however, was restricted to surrogate models satisfying homogeneous boundary conditions and therefore did not cover the full originally proposed objective.

Here, we complete the study of stiffness phenomena by extending the analysis to surrogate models with arbitrary nonzero Dirichlet boundary values. 
In particular, using an estimate of Lions and Magenes~\cite{Lions1972}, we combine an \(H^{-1}\) interior residual loss with an \(H^{1/2}\) boundary residual loss, yielding a formulation that is continuous in the \(H^1\) norm.

\subsection{Contribution}
Classically, a physics-informed neural network (PINN) is a neural network 
\begin{equation}
    u_{\bm \theta} : \Omega \subset \mathbb R^m \to \mathbb R
\end{equation}
parameterized by \(\bm\theta \in \mathbb R^d\), which is used to find an approximate solution of a PDE. Let \(m \in \mathbb N\), a PDE operator \(\mathcal N(\cdot)\) on the domain \(\Omega \subset \mathbb R^m\) with boundary \(\partial \Omega\) and a source term \(f(\cdot)\). Then, following the standard residual-based PINN formulation \cite{Raissi2019, Wang2021}, the PDE is directly included into the training by minimizing the loss functional
\begin{equation}
    \label{eq:l2_residual_loss_functional}
    \mathcal L(\bm \theta) 
    \coloneqq 
    \mathcal L_\Omega(\bm \theta) + \mathcal{L}_{\partial\Omega}(\bm \theta) 
    \coloneqq \int_\Omega \mathcal (\mathcal N(u_{\bm \theta})-f)^2 d^m \bm x + \int_{\partial \Omega}(u_{\bm \theta}\big|_{\partial\Omega} -g)^2d^{m-1} \bm x 
    \longrightarrow 
    \min.
\end{equation}
In practice, the integrals are approximated by finite sums over randomly sampled training data points \(P_\Omega \subset \Omega\) and \(P_{\partial\Omega} \subset \partial\Omega\) resulting in the discrete mean squared error (MSE) losses
\begin{equation}
    \mathcal L_\Omega(\bm \theta)
    \approx
    \frac{1}{|P_\Omega|}\sum_{\bm x \in P_\Omega} \mathcal (\mathcal N(u_{\bm \theta})(\bm x) -f(\bm x))^2, \\ 
    \qquad
    \mathcal L_{\partial \Omega}(\bm \theta) 
    \approx
    \frac{1}{|P_{\partial\Omega}|} \sum_{\bm x \in P_{\partial \Omega}}(u_{\bm \theta} (\bm x) -g(\bm x))^2.
\end{equation}
Specifically, the MSE loss may be dominated by the PDE residual, which can underweight the boundary terms and cause an imbalance in their magnitudes
\begin{equation}
    \mathcal{O}(\mathcal L_\Omega(\bm\theta)) \gg \mathcal{O}(\mathcal L_{\partial\Omega}(\bm\theta)).
\end{equation}
This stiffness phenomenon causes the gradient of the loss functional \(\mathcal L\) to be dominated by the contribution of the PDE residual, while the influence of the boundary residual becomes comparatively small.
As studied in several works \cite{Wang2021, Maddu2022, Suarez2024}, training faces slow convergence and instability, commonly appearing as boundary artifacts as well as limited accuracy compared to classic variational approaches.

Specifically, we detail our contribution below:

\begin{itemize}
    \item[\textbf{C1)}] We revisit discretizations of the \(H^{-1}\)-norm \cite{Suarez2024} and reinterpret them as the seminorms SINOS and D-SINOS. We derive continuity and coercivity estimates that characterize the conditioning of elliptic boundary value problems in Section~\ref{sec:subspace_induced_negative_order_sobolev_seminorm}, and Section~\ref{sec:dual_subspace_induced_negative_order_sobolev_seminorm}. The achieved bounds depend on Rayleigh quotients, and therefore, on the  chosen approximation space.

    \item[\textbf{C2)}] 
    In case of finite-dimensional polynomial spaces we show the Rayleigh quotients to imply continuity and coercivity estimates that improve conditioning in Section~\ref{sec:discretization_of_the_residual_loss_functional}, namely
    \begin{equation}
        \kappa_{L^2} \in \mathcal{O}(n^4), \qquad
        \kappa_{\mathrm{SINOS}} \in \mathcal{O}(n^3), \qquad
        \kappa_{\mathrm{D\text{-}SINOS}} \in \mathcal{O}(n^2),
    \end{equation}
    where \(n\) denotes the degree of freedom.
    
    \item[\textbf{C3)}] We demonstrate in Section~\ref{sec:random_sampling_test} that using a conjugate gradient (CG) method with an appropriate diagonal preconditioner requires only a few iterations in practice, making SINOS and D-SINOS computationally feasible, with quadratic runtime complexity.
    In lower spatial dimensions, a precomputed pseudoinverse can be used instead. 
    Moreover, domain decomposition into reference cells yields sparse operators that retain the effectiveness of their dense counterparts, see Section~\ref{sec:numerical_experiments}.

    \item[\textbf{C4)}] We apply SINOS and D-SINOS to PINNs in Section~\ref{sec:application_to_pinns}. Our experiments show that SINOS and D-SINOS improve the convergence behavior and prediction quality of PINNs for PDEs with elliptic behavior.
    Notably, SINOS converges faster than D-SINOS despite having a larger condition number, indicating that stiffness mitigation is important but does not fully determine PINN performance.
\end{itemize}
    
\subsection{Preliminaries}

Throughout this paper, a set \(\Omega \subset \mathbb R^m\) is called a domain, when it is nonempty, open, connected and Lebesgue measurable. Further, let \(\Omega \subset \mathbb R^m\) be a domain.
We use standard notation for Lebesgue and Sobolev spaces, following \cite{Adams2003, Brezis2011}.
We denote by \(L^2(\Omega)\) the space of equivalent classes of measurable functions \(u: \Omega \to \mathbb R\) such that
\begin{equation}
    \int_\Omega |u|^2 d^m\bm x 
    < 
    \infty,
\end{equation}
where functions are identified if they agree almost everywhere in \(\Omega\). The space \(L^2(\Omega)\) is equipped with the norm
\begin{equation}
    \|u\|_{L^2(\Omega)} 
    \coloneqq 
    \left( \int_{\Omega} |u|^2 d^m \bm x \right)^{1/2},
\end{equation}
which is induced by the inner product
\begin{equation}
    \langle u, v \rangle_{L^2(\Omega)} 
    \coloneqq 
    \int_\Omega u v d^m \bm x.
\end{equation}
Moreover, \(L^2(\Omega)\) is a Hilbert space. Further, we define
\begin{equation}
    L^\infty(\Omega) 
    \coloneqq 
    \left\{ u: \Omega \to \mathbb R \ : \ u \text{ is measurable}, \ \text{ess sup}_{\bm x \in \Omega} |u(\bm x)| < \infty \right\},
\end{equation}
endowed with the norm
\(
    \|u\|_{L^\infty(\Omega)} \coloneqq \text{ess sup}_{\bm x \in \Omega} |u|.
\)

For convenience, Table~\ref{tab:notation} serves as a complementary reference listing the notation and basic operations.

\begin{table}[ht]
    \centering
    \caption{Notation used throughout the paper.}
    \label{tab:notation}
    \renewcommand{\arraystretch}{1.15}
    \setlength{\tabcolsep}{8pt}
    \small
    \begin{tabular}{|c|l|c|l|}
    \hline
    $\Omega$ & domain &
    $\partial\Omega$ & boundary of the domain \\
    $m$ & spatial dimension &
    $\bm x$ & spatial variable \\
    $u_{\bm\theta}$ & neural-network approximation &
    $u_*$ & exact solution \\
    $\bm\theta$ & trainable parameters &
    $T$ & differential operator \\
    $f$ & source term &
    $g$ & Dirichlet boundary data \\
    $\operatorname{tr}(u)$ & trace of $u$ &
    $\mathcal L$ & residual loss functional \\
    $\mathcal L_\Omega$ & PDE loss &
    $\mathcal L_{\partial\Omega}$ & boundary loss \\
    $P_\Omega$ & interior grid &
    $P_{\partial\Omega}$ & boundary grid \\
    $V$ & approximation space &
    $\Pi_\square$ & polynomial space \\
    $\bm W$ & interior weight matrix &
    $\bm W_{\partial\Omega}$ & boundary weight matrix \\
    $\bm D^{(i)}$ & differentiation matrix &
    $\bm I$ & identity matrix \\
    $\bm J$ & SINOS matrix &
    $\bm K$ & D-SINOS matrix \\
    $|\cdot|_{-1,V}$ & SINOS seminorm &
    $|\cdot|_{-1,V^*}$ & D-SINOS seminorm \\
    $H^1(\Omega)$ & Sobolev space &
    $H^{-1}(\Omega)$ & negative Sobolev space \\
    $H^{1/2}(\partial\Omega)$ & boundary Sobolev space &
    $\kappa$ & condition number \\
    $\mathcal{O}(\cdot)$ & asymptotic upper bound &
    $\|\cdot\|_{L^2}$ & $L^2$-norm \\
    \hline
    \end{tabular}
\end{table}

\section{Sobolev Space Theory}

Sobolev spaces are the classic analytical setting for formulating and studying PDEs. We briefly recall some basics, including Sobolev spaces of positive integer, fractional, and negative integer-order. We refer to Adams and Fournier \cite{Adams2003} for further details.

\subsection{Sobolev spaces}

We begin with the elementary definition of positive integer-order Sobolev spaces.

\begin{definition}[Sobolev space]
    \label{def:sobolev_space}
    Let \(m, k \in \mathbb N\) with \(k \le m\) and \(\Omega \subset \mathbb R^m\) be a domain. The Sobolev space \(H^k(\Omega)\) is defined by
    \begin{equation}
        H^k(\Omega) 
        \coloneqq 
        \left\{u \in L^2(\Omega) \ : \ \partial_{\bm \alpha} u \in L^2(\Omega) \ \forall \bm \alpha \in \mathbb N_0^m: \ |\bm\alpha| \leq k\right\},
    \end{equation}
    endowed with the norm
    \begin{equation}
        \|u\|_{H^k(\Omega)} 
        = 
        \left(
            \int_{\Omega} |u|^2 d^m \bm x  
            + 
            \sum_{|\bm\alpha|\leq k} 
                \int_{\Omega} |\partial_{\bm\alpha}u|^2 d^m \bm x
        \right)^{1/2},
    \end{equation}
    where \(\partial_{\bm\alpha}\) corresponds to the weak derivative. The subspace \(H_0^k(\Omega) \subseteq H^k(\Omega)\) is defined as the closure of \(C_c^\infty(\Omega)\) in \(\|\cdot\|_{H^k(\Omega)}\), and is given by the space of functions in \(H^k(\Omega)\) whose traces on \(\partial \Omega\) vanish.
\end{definition}

\subsection{Sobolev space of fractional-order and the trace theorem}

We extend the definition of Sobolev spaces to the fractional-order Sobolev space \(H^{1/2}(\Omega)\), which was independently introduced by Aronszajn, Gagliardo and Slobodeckij~\cite{Aronszajn1955, Gagliardo1958, Slobodeckij1958}. This space is an intermediate space \(H^1(\Omega) \subset H^{1/2}(\Omega) \subset L^2(\Omega)\). 

\begin{definition}[Aronszajn, Gagliardo, Slobodeckij space]
    Let \(m \in \mathbb N\) and \(\Omega \subset \mathbb R^m\) be a domain. We define \(H^{1/2}(\Omega)\) as follows
    \begin{equation}
        H^{1/2}(\Omega) \coloneqq \left\{u\in L^2(\Omega) \ : \ \frac{|u(\bm x)-u(\bm y)|}{\|\bm x - \bm y\|_2^{(m+1)/2}} \in L^2(\Omega \times \Omega)\right\},
    \end{equation}
    endowed with the norm
    \begin{equation}
        \label{eq:5}
        \|u\|_{H^{1/2}(\Omega)} 
        \coloneqq 
        \left(
            \int_\Omega |u|^2 d^m \bm x 
            + 
            \int_\Omega \int_\Omega \frac{
                    |u(\bm x)-u(\bm y)|^2
                }{
                    \|\bm x - \bm y\|_2^{m+1}
                }
                d^m \bm x d^m \bm y
        \right)^{1/2}.
    \end{equation}
\end{definition}

We recall the following version of the trace operator, which is a result of Gagliardo~\cite{Gagliardo1957}.

\begin{theorem}[Trace operator]
    \label{theo:trace_operator}
    Let \(m \in \mathbb N\) and \(\Omega \subset \mathbb R^m\) be a bounded domain with Lipschitz boundary. Then, there exists a linear bounded and surjective operator \(\operatorname{tr}: H^1(\Omega) \to H^{1/2}(\partial\Omega)\) which extends the classic trace for all \(u \in H^1(\Omega) \cap C(\overline{\Omega})\)
    \begin{equation}
        \operatorname{tr}(u) = u\big|_{\partial\Omega}.
    \end{equation}
    Further, there exists a linear bounded extension operator \(E : H^{1/2}(\partial\Omega) \to H^1(\Omega)\) such that 
    \begin{equation}
        \operatorname{tr}(E(u)) 
        = 
        u.
    \end{equation}
\end{theorem}

\subsection{Sobolev space of negative-order for the dual space}

In classic numerical analysis of PDEs, the \(H^{-1}\) Sobolev space is a well-established concept \cite{Bochev97, Bramble97, Kim02}. 
It is typically defined as the dual space of \(H_0^1\). 
In the context of PINNs, boundary conditions are commonly enforced softly by adding boundary residual terms to the loss functional, rather than being satisfied by construction \cite{Berrone2023}.
As a consequence, the dual of \(H_0^1\) is inadequate, as functions with identical behavior in the interior but different boundary values may become indistinguishable when tested, which may lead to severe stiffness effects. 
These effects can be circumvented by instead considering the dual space of \(H^1(\Omega)\). 
Both definitions are well understood in the literature, see, \cite{Adams2003}.

This motivates the following definition.
\begin{definition}[Negative-order Sobolev spaces]
    Let the conditions of Definition~\ref{def:sobolev_space} be satisfied. We define the first negative-order Sobolev spaces as the duals of the positive-order spaces
    \begin{equation}
        H^{-1}(\Omega) 
        \coloneqq 
        H^1(\Omega)^*, 
        \qquad
        H_0^{-1}(\Omega) 
        \coloneqq 
        H_0^1(\Omega)^*,
    \end{equation}
    respectively endowed with the norms
    \begin{equation}
        \|u\|_{H^{-1}(\Omega)} 
        \coloneqq 
        \sup_{\substack{v \in H^1(\Omega):\, \|v\|_{H^1(\Omega)} = 1}} \langle u, v\rangle_{L^2(\Omega)}, 
        \qquad 
        \|u\|_{H_0^{-1}(\Omega)} 
        \coloneqq 
        \sup_{\substack{v \in H_0^1(\Omega):\, \|v\|_{H^1(\Omega)} = 1}} \langle u, v\rangle_{L^2(\Omega)}.
    \end{equation}
\end{definition}
The \(H^{-1}\)-space can, similarly to \(H^{1/2}\), be regarded as an intermediate space.
\begin{corollary}
    \label{cor:2}
    Let \(m \in \mathbb N\) and \(\Omega \subset \mathbb R^m\) be a domain. Then, the dual spaces satisfy the following inclusions
    \begin{equation}
        L^2(\Omega)^* \subseteq H^{-1}(\Omega) \subseteq H_0^{-1}(\Omega),
    \end{equation}
    and, for all \(u \in H^{-1}(\Omega)\), the associated norms obey
    \begin{equation}
        \label{eq:12}
        \|u\|_{H_0^{-1}(\Omega)} 
        \le 
        \|u\|_{H^{-1}(\Omega)} 
        \le 
        \|u\|_{L^2(\Omega)^*}.
    \end{equation}
\end{corollary}
The following example demonstrates that the inequalities in equation~\eqref{eq:12} can in fact be strict.
\begin{example}
    Let \(m \in \mathbb N\) and \(\Omega \subset \mathbb R^m\) be a bounded domain with Lipschitz boundary, with \(\eta\) denoting the outward pointing unit normal field. Define the linear functional 
    \begin{equation}
        \ell: 
        H^1(\Omega) 
        \to 
        \mathbb R, 
        \quad 
        v 
        \mapsto
        \int_{\partial\Omega} v \eta_1 d^{m-1}\bm x,
    \end{equation}
    where \(\eta_1\) is the first component of the outward unit normal vector on \(\partial\Omega\). Clearly, \(\ell\) vanishes on \(H_0^1(\Omega)\). Hence, \(\ell \in H_0^{-1}(\Omega)\) and coincides with the zero functional, i.e. \(\|\ell\|_{H_0^{-1}(\Omega)} = 0\). However, \(\ell\) is non-trivial on \(H^1(\Omega)\). Moreover, for \(v \in H^1(\Omega)\), by using integration by parts and the Cauchy-Schwarz inequality, we obtain 
    \begin{equation}
        \ell(v) 
        = 
        \int_\Omega \partial_1 v d^m \bm x \leq |\Omega|^{1/2} \|v\|_{H^1(\Omega)}.
    \end{equation}
    Hence, 
    \(
        \ell \in H^{-1}(\Omega)
    \) with
    \(
        0 < \|\ell\|_{H^{-1}(\Omega)} \le |\Omega|^{1/2}
    \). 
    Further,
    \(
        \ell \notin L^2(\Omega)^*
    \), since there exists a sequence 
    \(
        v_n \in C^\infty(\overline \Omega)
    \) with \(\|v_n\|_{L^2(\Omega)}=1\) which yields 
    \(
        |\ell(v_n)| \to \infty
    \) as \(n \to \infty\). Thus, 
    \(
        \|\ell\|_{L^2(\Omega)^*} = \infty
    \).
\end{example}
A key property of the \(H^{-1}\)-norm applied to the functional
\begin{equation}
    \label{eq:13}
    \ell(v) 
    \coloneqq 
    \int_\Omega \operatorname{div}(\bm F) v d^m\bm x, 
    \quad 
    v \in H^1(\Omega),
\end{equation}
is that it controls the divergence of the vector field \(\bm{F}\), but at the cost of an additional \(L^2(\partial\Omega)\)-norm term.
\begin{corollary}
    \label{cor:3}
    Let \(m \in \mathbb N\) and \(\Omega \subset \mathbb R^m\) be a bounded domain with Lipschitz boundary, with \(\eta\) denoting the outward pointing unit normal field of \(\partial\Omega\). Further, let a vector field \(\bm F = (F_1, \ldots, F_m)\) where \(F_i \in H^1(\Omega)\). Then, the following inequality holds
    \begin{equation}
        \|\ell\|_{H^{-1}(\Omega)} 
        \le 
        \sum_{i=1}^m \left(\|F_i\|_{L^2(\Omega)} + \|F_i\eta_i\|_{L^2(\partial\Omega)}\right),
    \end{equation}
    where \(\ell\) was defined in \eqref{eq:13}.
\end{corollary}
\begin{proof}
    Given \(v \in C^1(\Omega) \setminus \{0\}\), the statement is a direct consequence of a combination of the divergence theorem with the Cauchy-Schwarz inequality and the trace theorem
    \begin{equation}
        |\ell(v)|
        =
        \left|\int_\Omega (\bm{F} \cdot \nabla v )d^m \bm x - \int_{\partial\Omega} (\bm{F} \cdot \eta) v d^{m-1} \bm x \right|
        \leq
        \sum_{i=1}^m \left(\|F_i\|_{L^2(\Omega)}  + \|F_i \cdot \eta_i\|_{L^2(\partial\Omega)}\right) \|v\|_{H^1(\Omega)}.
    \end{equation}
    Divide both sides by \(\|v\|_{H^1(\Omega)} \neq 0\), take the supremum over all \(v \in C^1(\Omega) \setminus \{0\}\) to obtain the desired inequality by density of \(C^1(\Omega) \setminus \{0\}\) in \(H^1(\Omega)\).
\end{proof}

In conclusion, the \(H^{-1}(\Omega)\)-norm is the natural choice for capturing functionals acting on \(H^1(\Omega)\). This choice is particularly appropriate for least-squares formulations of second-order linear elliptic operators, as we will see next.

\section{PDE Theory}

In the following, we recall the classic functional-analytic setting of PDEs~\cite{Lions1972}. We revisit the well-posedness of the associated loss functional by summarizing standard results in the infinite-dimensional setting that guarantee coercivity and continuity of the minimization problem, thereby providing the foundation for subsequent discretizations.

\subsection{Second-order linear elliptic boundary value problem}

Second-order linear elliptic boundary value problems serve as a fundamental class, because of their wide range of applications and because they often appear as subproblems in more complex systems \cite{Jost2007, Brezis2011}.

We start with the following definition.

\begin{definition}[Second-order linear elliptic boundary value problems]
    \label{def:bvp}
    Let \(m \in \mathbb N\) and \(\Omega \subset \mathbb R^m\) be a bounded domain with Lipschitz boundary, where \(\eta\) denotes the outward pointing unit normal field of \(\partial\Omega\). Further, let \(\bm A \in L^\infty(\Omega)^{m \times m}\) be uniformly elliptic, that is, there exist constants \(C_0, C_1 > 0\) such that for all vectors \(\bm y \in \mathbb R^m\)
    \begin{equation}
        \label{eq:7}
        C_0 \bm y^\top \bm y \leq \bm y^\top \bm A(\bm x) \bm y \leq C_1 \bm y^\top \bm y,
    \end{equation}
    for almost all \(\bm x \in \Omega\). Moreover, given \(u \in H^1(\Omega)\) with \(\eta \cdot \bm A \nabla u \in L^2(\partial\Omega)\), we reformulate the definition of \(\operatorname{div}(\bm A \nabla u)\) by means of \emph{Green's formula} in the distributional sense by
    \begin{equation}
        \label{eq:8}
        \int_\Omega \operatorname{div}(\bm{A} \nabla u) v d^m\bm x 
        = 
        \int_\Omega - \bm{A} \nabla u \cdot \nabla v d^m\bm x  + \int_{\partial\Omega} (\eta \cdot \bm{A}\nabla u) v d^{m-1}\bm x ,
    \end{equation}
    for all test functions \(v \in H^1(\Omega) \). We consider the \(m\)-dimensional inhomogeneous second-order linear elliptic Dirichlet boundary value problem in \(\Omega\) for \(f \in L^2(\Omega)\) and \(g \in H^{1/2}(\partial\Omega)\). That is finding a solution \(u \in H^1(\Omega)\) such that 
    \begin{equation}
        \label{eq:bvp}
         \begin{cases}
            Tu 
            &= 
            f \text{ in } \Omega, \\
            u 
            &= 
            g \text{ on } \partial\Omega,
        \end{cases}
    \end{equation}
    where the operator \(T\) is given by
    \begin{equation}
        T: H^1(\Omega) \to H^{-1}(\Omega), 
        \quad 
        u \mapsto \operatorname{div} (\bm{A} \nabla u) + \bm{b} \cdot \nabla u + cu,
    \end{equation}
    with \(\bm b \in L^\infty(\Omega)^m\), and \(c \in L^\infty(\Omega)\).
\end{definition}

\subsection{Residual loss functional}
We reformulate the Dirichlet boundary value problem in \eqref{eq:bvp} as a global optimization problem by introducing a soft-constrained PDE residual loss functional.
\begin{definition}
    \label{def:residual_loss_functional}
    Let the assumptions of Definition~\ref{def:bvp} be satisfied. We define the residual loss functional \(\mathcal L: H^1(\Omega) \to \mathbb R_0^+\) as
    \begin{equation}
        \label{eq:residual_loss_functional}
        \mathcal L(u) 
        = 
        \frac{1}{2} \|Tu-f\|_{H^{-1}(\Omega)}^2 + \frac{1}{2} \|\operatorname{tr}(u) - g\|_{H^{1/2}(\partial\Omega)}^2.
    \end{equation}
\end{definition}
The loss measures the squared PDE residual together with a penalty term enforcing the boundary condition in a weak sense, yielding a non-negative quadratic functional. Under suitable assumptions, it can be shown that the global minimizer of this functional corresponds to the solution of the underlying boundary value problem.

\subsection{Well-posedness}
The following result originates from Lions and Magenes \cite{Lions1972} and can also be reformulated in higher-order Sobolev spaces.
\begin{theorem}[Lions and Magenes]
    \label{theo:lions_and_magenes}
    Let the assumptions of Definition~\ref{def:bvp} be satisfied and assume that every solution to \eqref{eq:bvp} is unique in \(H^1(\Omega)\). That is, if \(u \in H^1(\Omega)\) such that
    \begin{equation}
        \|Tu\|_{H^{-1}(\Omega)} + \|\operatorname{tr}(u)\|_{H^{1/2}(\partial\Omega)} = 0,
    \end{equation}
    it implies \(\|u\|_{H^1(\Omega)} = 0\). Then, there exist constants \(C_0=C_0(\Omega, \bm{b}, c), C_1=C_1(\Omega, \bm{b}, c)>0\), depending on the domain and the data, such that for all \(u \in H^1(\Omega)\)
    \begin{equation}
        \label{eq:9}
        C_0 \|u\|_{H^1(\Omega)} \leq \|Tu\|_{H^{-1}(\Omega)} + \|\operatorname{tr}(u)\|_{H^{1/2}(\partial\Omega)} \leq C_1 \|u\|_{H^1(\Omega)}.
    \end{equation}
\end{theorem}
% \begin{proof}
%     We provide a mainly self-contained proof of the theorem in the stated form in Appendix~\ref{app:lions_and_magenes}.
% \end{proof}
As a consequence, the residual loss functional from equation~\eqref{eq:residual_loss_functional} is equivalent to the squared \(H^1\)-norm on \(\Omega\). In particular, Theorem~\ref{theo:lions_and_magenes} delivers a two-sided error control for the residual loss functional of Definition~\ref{def:residual_loss_functional}, establishing well-posedness of the associated minimization problem in \(H^1(\Omega)\).
\begin{corollary}[Coercivity and Continuity]              
    \label{cor:residual_loss_functional}
    Let the assumptions of Definition~\ref{def:residual_loss_functional} be satisfied and \(u_* \in H^1(\Omega)\) a solution to \eqref{eq:bvp}. Then, \(\mathcal L\) is a quadratic form satisfying for all \(u \in H^1(\Omega)\)
    \begin{equation}
        C_0 \|u - u_*\|_{H^1(\Omega)}^2 \leq \mathcal L(u) \leq C_1\|u - u_*\|_{H^1(\Omega)}^2,
    \end{equation}
    with constants \(C_0 = C_0(\Omega, \bm{b}, c), C_1 = C_1(\Omega, \bm{b}, c) > 0\), depending on the domain and the data.
\end{corollary}

\section{Approximations of the negative-order Sobolev norm}
\label{sec:approximations_of_the_negative_order_sobolev_norm}
In this section, we define two seminorms obtained by restricting the negative-order Sobolev norm to suitable subspaces. These are the subspace-induced negative-order Sobolev seminorm (SINOS) in Section~\ref{sec:subspace_induced_negative_order_sobolev_seminorm} and the dual subspace-induced negative-order Sobolev seminorm (D-SINOS) in Section~\ref{sec:dual_subspace_induced_negative_order_sobolev_seminorm}, cf. \cite{Suarez2024}.

Our aim is to approximate the \(H^{-1}\)-norm by means of these subspaces to mitigate stiffness effects arising in PINNs. 
We establish coercivity constants that hold for arbitrary elements and continuity constants with respect to the gradient and the Laplace operator. 

Although we obtain a valid upper bound for SINOS, it generally overestimates boundary terms. In contrast, D-SINOS induces a rescaling under which the gradient is continuous with constant one in the \(L^2\)-norm, and the Laplace operator is continuous with constant one in the \(H^{1}\)-seminorm. Later in Section~\ref{sec:discretization_of_the_residual_loss_functional}, computing the associated Rayleigh quotients shows that the rescaling of D-SINOS improves the behavior of the operator rather than simply re-normalizing the coercivity constant.

\subsection{Subspace-induced negative-order Sobolev seminorm}
\label{sec:subspace_induced_negative_order_sobolev_seminorm}
We prove a uniform lower bound in Proposition~\ref{prop:coercivity}, showing that boundary terms are not suppressed by the choice of the present seminorm. Proposition~\ref{prop:continuity} shows that the Laplace operator fails to be uniformly continuous, reflecting that the contribution of boundary terms may lead to an amplification.

Let \(m \in \mathbb N\) be the spatial dimension and let \(\Omega \subset \mathbb R^m\) be a bounded domain. We introduce the following subspace-induced seminorm on \(H^{-1}(\Omega)\).
\begin{definition}
    \label{def:seminorm}
    Let \(V \subset H^1(\Omega)\) be a closed linear subspace. For \(u \in H^{-1}(\Omega)\), we define the standard first negative-order Sobolev seminorm associated with \(V\) by
    \begin{equation}
        |u|_{-1, V} \coloneqq \sup_{v \in V: \, \|v\|_{H^1(\Omega)} = 1} \langle u, v\rangle_{L^2(\Omega)}.
    \end{equation}
\end{definition}

By definition, \(|u|_{-1, V}\) is obtained by taking the supremum over test functions in \(V\). Since \(V \subset H^1(\Omega)\), this supremum is taken over a smaller space than in the definition of \(\|u\|_{H^{-1}(\Omega)}\), and therefore yielding
\(
    |u|_{-1, V} \leq \|u\|_{H^{-1}(\Omega)}.
\) 
Even if \(u \in V\), the full norm \(\|u\|_{H^{-1}(\Omega)} \) generally differs from \(|u|_{-1, V}\), making the inequality in many cases strict. This motivates the following decomposition into the component associated with the subspace \(V\) and the component associated with its \(H^1\)-orthogonal complement \(V^\perp\). 
\begin{theorem}[Orthogonal decomposition]
    \label{theo:negative_norm_decomposition}
    Let \(V \subset H^1(\Omega)\) be a closed linear subspace and denote by \(V^\perp\) its orthogonal complement in \(H^1(\Omega)\). Then, for \(u \in H^1(\Omega)\), it holds that
    \begin{equation}
        \|u\|_{H^{-1}(\Omega)} = \sqrt{|u|_{-1, V}^2 + |u|_{-1, V^\perp}^2}.
    \end{equation}
\end{theorem}
\begin{proof}
    We refer to Appendix~\ref{app:negative_norm_decomposition} for the proof.
\end{proof}
Next, we establish a lower bound, depending on specific Rayleigh quotients, for the seminorm \(|\cdot|_{-1,V}\) in terms of the \(H^{-1}\)-norm.
\begin{proposition}[Coercivity]
    \label{prop:coercivity}
    Let \(V \subset H^1(\Omega)\) be a closed linear subspace and denote by \(V^\perp\) its orthogonal complement in \(H^1(\Omega)\). Then, for \(u \in H^1(\Omega)\) it holds that
    \begin{equation}
        \label{eq:14}
        \|u\|_{H^{-1}(\Omega)} \le |u|_{-1, V} + \left(\sup_{w \in V^\perp \setminus \{0\}} \frac{\|w\|_{L^2(\Omega)}}{\|w\|_{H^1(\Omega)}}\right) \|u\|_{L^2(\Omega)}.
    \end{equation}
    Further, if \(u \in V\), then
    \begin{equation}
        \label{eq:15}
        \|u\|_{H^{-1}(\Omega)} 
        \le 
        \left(1 + \sup_{v \in V \setminus \{0\}} \frac{\|v\|_{H^1(\Omega)}}{\|v\|_{L^2(\Omega)}} \sup_{w \in V^\perp \setminus \{0\}} \frac{\|w\|_{L^2(\Omega)}}{\|w\|_{H^1(\Omega)}}\right) |u|_{-1, V}.
    \end{equation}
\end{proposition}
\begin{proof}
    Let \(u \in H^1(\Omega)\). Applying Theorem~\ref{theo:negative_norm_decomposition} and using the fact that \(|\sqrt{a} - \sqrt{b}| \leq |\sqrt{a-b}|\) for any \(a>b>0\), yields 
    \begin{equation}
        \left|\|u\|_{H^{-1}(\Omega)} - |u|_{-1, V}\right| 
        = 
        \left|\sqrt{|u|_{-1, V}^2 + |u|_{-1, V^\perp}^2} - \sqrt{|u|_{-1, V}^2} \right| 
        \le 
        |u|_{-1,V^\perp}.
    \end{equation}
    Further, by recalling Definition~\ref{def:seminorm} and applying the Cauchy-Schwarz inequality, we obtain
    \begin{equation}
        |u|_{-1, V^\perp} 
        = 
        \sup_{w \in V^\perp \setminus \{0\}} \frac{\langle u, w \rangle_{L^2(\Omega)}}{\|w\|_{H^1(\Omega)}} 
        \le 
        \left(\sup_{w \in V^\perp \setminus \{0\}} \frac{\|w\|_{L^2(\Omega)}}{\|w\|_{H^1(\Omega)}}\right) \|u\|_{L^2(\Omega)},
    \end{equation}
    completing the proof of the first inequality~\eqref{eq:14}.
    Now, let \(u \in V\), then
    \begin{equation}
        \|u\|_{L^2(\Omega)} 
        = 
        \sup_{v \in V \setminus\{0\}} \frac{
            \langle u, v\rangle_{L^2(\Omega)}
        }{
            \|v\|_{L^2(\Omega)}
        } 
        \le 
        \left(\sup_{v \in V \setminus \{0\}} \frac{
            \|v\|_{H^1(\Omega)}
        }{
            \|v\|_{L^2(\Omega)}
        }\right) |u|_{-1, V},
    \end{equation}
    yields the second inequality~\eqref{eq:15} and completes the proof.
\end{proof}
We conclude with the following proposition, bounding the Laplace operator in \(|\cdot|_{-1, V}\).
\begin{proposition}[Stiff continuity of the Laplace operator]
    \label{prop:continuity}
    Let \(V \subset H^1(\Omega)\) be a closed linear subspace. Then, for \(u \in V \cap H^2(\Omega)\), it holds
    \begin{equation}
        |\Delta u|_{-1, V} \lesssim \|u\|_{H^2(\Omega)}.
    \end{equation}
\end{proposition}
\begin{proof}
    Let \(u \in V \cap H^2(\Omega)\). Then, in the distributional sense, as in Definition~\ref{def:bvp}, for any \(v \in H^1(\Omega)\)
    \begin{equation}
        \int_\Omega (\Delta u) v d^m \bm x = B(u, v) - \int_\Omega \nabla u \cdot \nabla v d^m \bm x, \qquad B(u, v) \coloneqq \int_{\partial \Omega} (\nabla u \cdot \eta) v d^{m-1} \bm x.
    \end{equation}
    Thus, we estimate
    \begin{equation}
        |\Delta u|_{-1, V} 
        \le 
        |u|_{H^1(\Omega)} 
        + 
        \sup_{v \in V \setminus \{0\}} \frac{
            B(u, v)
        }{
            \|v\|_{H^1(\Omega)}
        }.
    \end{equation}
    Then, the sharpest upper bound of the boundary term is
    \begin{equation}
        |B(u, v)| 
        \le
        \|\nabla u \cdot \eta\|_{L^2(\partial\Omega)} \|v\|_{L^2(\partial\Omega)} 
        \lesssim 
        \|u\|_{H^2(\Omega)} \|v\|_{H^1(\Omega)}.
    \end{equation}
    Consequently, we obtain 
    \begin{equation}
        |\Delta u|_{-1, V} 
        \lesssim 
        |u|_{H^1(\Omega)} 
        + 
        \|u\|_{H^2(\Omega)} 
    \end{equation}
\end{proof}
As will be presented later in Section~\ref{sec:numerical_experiments}, numerical experiments verify that \(|\cdot|_{-1,V}\) fails to render the Laplace operator \(\Delta u\) uniformly continuous on \(V\) endowed with the \(H^1\)-norm. Consequently, the residual loss functional from Definition~\ref{def:residual_loss_functional} turns out to be stiff. This motivates us to introduce the following norm that, in theory, mitigates this problem.

\subsection{Dual subspace-induced negative-order Sobolev seminorm}
\label{sec:dual_subspace_induced_negative_order_sobolev_seminorm}
We prove a uniform upper bound of the Laplace operator in Proposition~\ref{prop:dual_continuity} preventing boundary terms from blowing up. Proposition~\ref{prop:dual_coercivity} reflects that certain contributions are still deliberately down-weighted.

Let \(m \in \mathbb N\) be the spatial dimension and let \(\Omega \subset \mathbb R^m\) be a bounded domain. We introduce the following dual subspace-induced seminorm on \(H^{-1}(\Omega)\).
\begin{definition}
    \label{def:dual_seminorm}
    Let \(V \subset H^1(\Omega)\) be a finite-dimensional linear subspace which is closed under partial differentiation. For each coordinate, denote by 
    \(
        \partial_i : 
        \left(V, {\|\cdot\|_{L^2(\Omega)}}\right) 
        \to 
        \left(V, {\|\cdot\|_{L^2(\Omega)}}\right)
    \)
    the \(i\)-th partial derivative operator and by 
    \(
        \partial_i^* : 
        \left(V, {\|\cdot\|_{L^2(\Omega)}}\right) 
        \to 
        \left(V, {\|\cdot\|_{L^2(\Omega)}}\right)
    \)
    its adjoint. For \(u \in H^{-1}(\Omega)\), we define the first negative-order Sobolev seminorm associated with \(V^*\) by
    \begin{equation}
        |u|_{-1, V^*} \coloneqq \sup_{v \in V: \|v\|_{1, V^*} = 1} \langle u, v \rangle_{L^2(\Omega)},
    \end{equation}
    where \(\|\cdot\|_{1, V^*}\) is a modified \(H^1\)-norm on the subspace \(V \simeq V^*\) given, for \(v \in V\), by
    \begin{equation}
        \|v\|_{1, V^*} \coloneqq \|(v, \partial_1^*v, \ldots, \partial_m^*v)\|_{L^2(\Omega)},
    \end{equation}
    induced, for \(v, w \in V\), by the inner product
    \begin{equation}
        \langle v, w\rangle_{1, V^*} \coloneqq \langle v, w\rangle_{L^2(\Omega)} + \sum_{i=1}^m\langle\partial_i^*v,\partial_i^*w\rangle_{L^2(\Omega)}.
    \end{equation}
\end{definition}
We begin by analyzing the behavior of the Laplace operator under the seminorm \(|\cdot|_{-1, V^*}\). The following proposition establishes the \emph{essential continuity properties} by which stiffness regarding the residual loss functional~\eqref{eq:residual_loss_functional} can be controlled from above.
\begin{proposition}[Continuity of the Laplace operator]
    \label{prop:dual_continuity}
    Let \(V \subset H^1(\Omega)\) be a finite-dimensional linear subspace which is closed under partial differentiation. For \(u \in V\), the seminorm \(|\cdot|_{-1, V^*}\) satisfies
    
    \begin{enumerate}[label=(\roman*)]
        \item \(
            |\partial_i u|_{-1, V^*} \leq \| u\|_{L^2(\Omega)},
        \)
        \item \(
            |\Delta u|_{-1, V^*} \leq |u|_{H^1(\Omega)}.
        \)
    \end{enumerate}
\end{proposition}
\begin{proof}
    Let \(u \in V\). We begin to demonstrate assertion \emph{(i)}
    \begin{equation}
        |\partial_i u|_{-1, V^*} 
        = 
        \sup_{v \in V: \|v\|_{1, V^*} = 1} \langle u, \partial_i^* v  \rangle_{L^2(\Omega)} \\
        \le
        \sup_{v \in V: \|v\|_{1, V^*} = 1} \|u\|_{L^2(\Omega)}  \|\partial_i^*v\|_{L^2(\Omega)} \\
        \le
        \|u\|_{L^2(\Omega)}.
    \end{equation}
    Assertion \emph{(ii)} follows by an analogous computation 
    \begin{equation}
        |\Delta u|_{-1, V^*} = \sup_{v \in V : \|v\|_{1, V^*} = 1} \langle \Delta u, v\rangle_{L^2(\Omega)} = \sup_{v \in V: \|v\|_{1, V^*}=1} \sum_{i=1}^m \langle \partial_i u, \partial_i^* v\rangle_{L^2(\Omega)},
    \end{equation}
    and, therefore
    \begin{equation}
        |\Delta u|_{-1, V^*} \leq \sup_{v \in V : \|v\|_{1, V^*}=1}\sum_{i=1}^m \| \partial_i u\|_{L^2(\Omega)} \|\partial_i^* v\|_{L^2(\Omega)} \leq \sup_{v \in V : \|v\|_{1,V^*} = 1} |u|_{H^1(\Omega)} \|v\|_{1, V^*} = |u|_{H^1(\Omega)}.
    \end{equation}
\end{proof}
Neuberger introduced and systematically developed the concept of Laplacians \cite{Parima2008, Neuberger97}, building on earlier ideas of Beurling \cite{Beurling59}. In this context, Laplacians are defined as the adjoints of embeddings of the form \(H^1 \hookrightarrow L^2\), which specifies in our case as follows.
\begin{proposition}[Laplacian]
    \label{prop:laplacian}
    Let \(V \subset H^1(\Omega)\) be a finite-dimensional linear subspace that is closed under partial differentiation. Then, the Laplacian of \(|\cdot|_{1, V^*}\) is given by
    \begin{equation}
        M = \left( \operatorname{id}_V + \sum_{i=1}^m \partial_i \partial_i^* \right)^{-1},
    \end{equation}
    yielding
    \(
        |u|_{-1, V^*} = \|M u\|_{1, V^*}
    \)
    for all \(u \in V\).
\end{proposition}
\begin{proof}
    % We defer the proof to Appendix~\ref{app:laplacian}.
    A proof is given in \cite[Proposition~2]{Suarez2024}
\end{proof}
We present a finite-dimensional analogue of \cite[Theorem~3.9, p.~62]{Adams2003} that yields a coercivity estimate for the seminorm \(|\cdot|_{-1,V^*}\) in Proposition~\ref{prop:dual_coercivity}.
\begin{theorem}[Infimum representation]
    \label{theo:infimum_representation}
    Let \(V \subset H^1(\Omega)\) be a finite-dimensional linear subspace that is closed under partial differentiation. Then, for \(u \in V\), the dual subspace-induced first negative-order Sobolev seminorm admits the following representation
    \begin{equation}
        \label{eq:3}
        |u|_{-1, V^*} = \inf\left\{ \|(v_0, v_1, \ldots, v_m)\|_{L^2(\Omega)} \ : \ u = v_0 + \sum_{i=1}^m \partial_i v_i, \  v_i \in V\right\}.
    \end{equation}
\end{theorem}
\begin{proof}
    A proof can be found in Appendix~\ref{app:infimum_representation}.
\end{proof}
Finally, we present the coercivity result for \(|\cdot|_{-1,V^*}\).
\begin{proposition}[Coercivity]
    \label{prop:dual_coercivity}
    Let \(V \subset H^1(\Omega)\) be a finite-dimensional linear subspace that is closed under partial differentiation, and let
    \(
        \operatorname{tr}: (V, \|\cdot\|_{H^1(\Omega)}) \to L^2(\partial\Omega)
    \)
    denote the trace operator. Then, for \(u \in V\), the seminorm \(|\cdot|_{-1,V^*}\) satisfies
    \begin{equation}
        \|u\|_{H^{-1}(\Omega)} 
        \leq  
        2 \sqrt{m+1} \max \left\{1, C \left(\sup_{v \in V \setminus \{0\}} \frac{\|v\|_{H^1(\Omega)}}{\|v\|_{L^2(\Omega)}}\right) \right\} |u|_{-1, V^*},
    \end{equation}
    with a constant \(C = C(\Omega)> 0\) only depending on the domain.
\end{proposition}
\begin{proof}
    Let \(u \in V\). Further, let a decomposition \(v_0, v_1, \ldots, v_m \in V\) of \(u\) in the sense of \eqref{eq:3}. Then, we can estimate the following
    \begin{align}
        \frac{1}{m+1}\|u\|_{H^{-1}(\Omega)}^2 
        & \le 
        \|v_0\|_{H^{-1}(\Omega)}^2 + \sum_{i=1}^m \|\partial_i v_i\|_{H^{-1}(\Omega)}^2 \\
        & \le 
        \|v_0\|_{L^2(\Omega)}^2 + \sum_{i=1}^m \left(\|v_i\|_{L^2(\Omega)} + \|v_i\|_{L^2(\partial\Omega)}\right)^2 \\
        & \leq 
        \|v_0\|_{L^2(\Omega)}^2 + 2 \sum_{i=1}^m \left(\|v_i\|_{L^2(\Omega)}^2 + \|\operatorname{tr}\|^2 \cdot \|v_i\|_{H^1(\Omega)}^2 \right) \\
        & \leq 
        4 \max \left\{1, C^2 \left(\sup_{v \in V \setminus \{0\}} \frac{\|v\|_{H^1(\Omega)}}{\|v\|_{L^2(\Omega)}}\right)^2\right\} \|(v_0, v_1, \ldots, v_m)\|_{L^2(\Omega)}^2.
    \end{align}
    Here, the constant depends on the norm of the trace operator which itself depends only on the domain \(\Omega\) and is therefore absorbed into a generic constant. Notably, the left-hand side is independent of the particular choice of the decomposition \(v_0, v_1, \ldots, v_m\). Taking the infimum over all such decompositions on the right-hand and applying Theorem \ref{theo:infimum_representation}, we obtain
    \begin{equation}
        \|u\|_{H^{-1}(\Omega)}^2 
        \leq 
        4 (m+1)\max \left\{1, C^2 \left(\sup_{v \in V \setminus \{0\}} \frac{\|v\|_{H^1(\Omega)}}{\|v\|_{L^2(\Omega)}}\right)^2\right\} |u|_{-1, V^*}^2.
    \end{equation}
    Finally, taking the square root of both sides completes the proof.
\end{proof}
We proceed by discretizing with specific choices for the finite-dimensional linear subspace \(V\).

\section{Discretization of the residual loss functional}
\label{sec:discretization_of_the_residual_loss_functional}
In this section, we specify the choice of the finite-dimensional subspace \(V\) to be a polynomial space, well-suited for the non-periodic setting of PINNs \cite{Hecht2026, Cohen2017}.
We then investigate the resulting discrete Sobolev norms and compute the Rayleigh quotients required to apply the continuity and coercivity estimates. This yields conditioning improvements of one order of magnitude for SINOS and two orders of magnitude for D-SINOS compared to the \(L^2\)-based formulation.

\subsection{Polynomial space}
\label{sec:polynomial_space}
The construction of D-SINOS is based on finite-dimensional subspaces \(V \subset H^1(\Omega)\) that are closed under partial differentiation. \(V\) is uniquely determined by a finite basis 
\begin{equation}
    v_1, \ldots, v_N \in H^1(\Omega), \quad V = \operatorname{span}_{\mathbb R}\{v_1,\dots,v_N\}
\end{equation}
such that for each coordinate direction \(j\) and each basis function \(v_i\), \(\partial_j v_i \in V\).

A canonical choice is given by multivariate polynomial spaces. Let \(m \in \mathbb N\), Let \(n_1, \ldots, n_m \in \mathbb N_0\), and \(\Omega = [-1, 1]^m\). We define the maximal degree polynomial space
\begin{equation}
    \Pi_\square 
    \coloneqq 
    \operatorname{span}_{\mathbb R}\{ \bm x^{\bm k} = x_1^{k_1} \cdots x_m^{k_m} \ : \ \bm k \in \square\},
\end{equation}
where
\begin{equation}
    \square \coloneqq \{0,\ldots,n_1\}\times\cdots\times\{0,\ldots,n_m\}.
\end{equation}
The space \(\Pi_\square\) is finite-dimensional and closed under partial differentiation.

\subsection{Sobolev cubature rules}
\label{sec:sobolev_cubature_rules}
Next, we specify a reference grid \(\bm P_{\Omega} \subset \Omega\) and Sobolev cubature rules for computing SINOS and D-SINOS. We construct a tensorial grid from a set of one-dimensional Gauss-Legendre nodes and derive the corresponding Sobolev cubature rules, cf. \cite{Suarez2024}. The nodes and quadrature weights of the one-dimensional Gauss–Legendre rule are described in the standard literature, see e.g. \cite[pp.~436-438]{Quarteroni2007}.

\paragraph{The $L^2(\Omega)$ norm:} Let \(m \in \mathbb N\) and \(\Omega = [-1, 1]^m\). Let
\(
    P_i = \{p_{0,i}, p_{1,i}, \dots, p_{n_i,i}\}\subset [-1,1]
\)
denote the set of one-dimensional Gauss-Legendre nodes for each dimension, \(n_i \in \mathbb N\). We define the \emph{tensorial grid}
\begin{equation}
    P_\Omega \coloneqq \{\bm p_{\bm k} \coloneqq (p_{k_1,1}, p_{k_2,2}, \ldots, p_{k_m,m}) \ : \ \bm k \in \square\}.
\end{equation}
Given \(n_i \coloneqq \max_{\bm k \in \square} k_i\), let \(w_{0,i}, w_{1,i}, \ldots, w_{n_i,i}\) denote the Gauss-Legendre weights of \(P_i\). For an index \(\bm k \in \square\), the Gauss-Legendre quadrature weight is then given by
\begin{equation}
    w_{\bm k}
    \coloneqq
    w_{k_1} \dots w_{k_m},
\end{equation}
defining the diagonal weight matrix
\(
    \bm W = \text{diag}(w_{\bm k})_{\bm k \in \square}.
\)
For \(u \in H^1(\Omega)\), we denote by \(\bm u\) the vector of the values of \(u\) in the points \(P_\Omega\) in lexicographic order. Then, the \(L^2\) norm can be approximated by
\begin{equation}
    \|u\|^2_{L^2(\Omega)} 
    \approx 
    \sum_{\bm k \in \square} w_{\bm k} u(\bm p_{\bm k})^2
    = 
    \bm u^\top \bm W \bm u
\end{equation}

\paragraph{The $H^1(\Omega)$ norm:} Let \(\bm D^{(i)}\in \mathbb R^{|\square| \times |\square|}\) denote the differentiation matrix corresponding to the \(i\)-th partial derivative of a multivariate interpolating polynomial in \(P_\Omega\). For \(\bm k \in \square\), it is defined via
\begin{equation}
    (\bm D^{(i)} \bm u)_{\bm k} \coloneqq \left(\partial_i Q_{\bm u}\right)(\bm p_{\bm k}),
\end{equation}
where \(Q_{\bm u}\) is the interpolating polynomial satisfying \(Q_{\bm u}(\bm p_{\bm l}) = u_{\bm l}\) for all \(\bm l \in \square\). Define the \(\bm W\)-self-adjoint matrices
\begin{equation}
    \bm J = \bm I + \sum_{i=1}^m \bm W^{-1} (\bm D^{(i)})^\top \bm W  \bm D^{(i)}, \qquad \bm K = \bm I + \sum_{i=1}^m \bm D^{(i)} \bm W^{-1} (\bm D^{(i)})^\top \bm W,
\end{equation}
with \(\bm I\) being the identity matrix. These matrices represent Sobolev cubature rules which, for \(u \in V\), yield the norms \(\|u\|_{1, V}\) and \(\|u\|_{1, V^*}\), respectively
\begin{equation}
    \label{eq:discrete_quadratic_forms}
    \|u\|_{1,V}^2 
    = 
    \bm u^\top \bm W \bm J \bm u
    \qquad 
    \|u\|_{1, V^*}^2
    = 
    \bm u^\top \bm W \bm K \bm u.
\end{equation}

\paragraph{SINOS and D-SINOS:} A direct consequence of Proposition~\ref{prop:laplacian}, is that SINOS and D-SINOS can be computed by inverting \(\bm J\) and \(\bm K\), respectively, for \(u \in V\) as follows
\begin{equation}
    \label{eq:discrete_quadratic_forms_negative_norm}
    \|u\|_{-1,V}^2 
    = 
    \bm u^\top \bm W \bm J^{-1} \bm u, 
    \qquad 
    \|u\|_{-1, V^*}^2
    = 
    \bm u^\top \bm W \bm K^{-1} \bm u.
\end{equation}

\paragraph{The $H^{1/2}(\partial\Omega)$ norm:}
On each boundary face
\(
    \partial\Omega_{j,\sigma}
    \coloneqq
    \{\bm x\in \partial\Omega : x_j=\sigma\},
\)
\(
    j=1,\ldots,m,
\)
\(
    \sigma \in \{-1, 1\}
\)
we fix the \(j\)-th coordinate to \(\sigma\) and use the Gauss-Legendre tensorial grid the remaining coordinates.
This yields the boundary grid
\begin{equation}
    P_{\partial\Omega}
    =
    \{\bm x_1,\ldots,\bm x_{N_{\partial\Omega}}\},
\end{equation}
with corresponding cubature weights \(\omega_1,\ldots,\omega_{N_{\partial\Omega}}\). We set
\begin{equation}
    \bm W_{\partial\Omega}
    \coloneqq
    \operatorname{diag}
    (\omega_1,\ldots,\omega_{N_{\partial\Omega}}).
\end{equation}
For \(\varphi \in H^{1/2}(\partial\Omega)\), we denote by \(\bm \varphi\) the vector of values of \(\varphi\) in the points \(P_{\partial\Omega}\) in lexicographic order.

We approximate the \(H^{1/2}(\partial\Omega)\) norm~\cite{Gagliardo1958, Adams2003} by
\begin{equation}
    \|\varphi\|_{H^{1/2}(\partial\Omega)}^2
    \approx
    \bm\varphi^\top
    \bm W_{\partial\Omega}
    \bm\varphi
    +
    \sum_{\substack{i,j=1, \ i \neq j}}^{N_{\partial\Omega}}
    \omega_i \omega_j
    \frac{
        \left(\varphi(\bm x_i)-\varphi(\bm x_j)\right)^2
    }{
        |\bm x_i-\bm x_j|^{m}
    }.
\end{equation}
Equivalently, this can be written as
\begin{equation}
    \|\varphi\|_{H^{1/2}(\partial\Omega)}^2
    \approx
    \bm\varphi^\top \, 
    \bm W_{1/2,\partial\Omega}
    \, \bm\varphi,
    \qquad
    \bm W_{1/2,\partial\Omega}
    \coloneqq
    \bm W_{\partial\Omega}
    +
    2 \bm L_{\partial\Omega}.
\end{equation}
Here, \(\bm L_{\partial\Omega}\) is the graph Laplacian of the interaction matrix \(\bm B_{\partial\Omega}\), defined by
\begin{equation}
    (\bm B_{\partial\Omega})_{i,j}
    \coloneqq
    \begin{cases}
        \omega_i \omega_j |\bm x_i - \bm x_j|^{-m}, & i \neq j,\\
        0, & i=j,
    \end{cases}
\end{equation}
and
\begin{equation}
    \bm L_{\partial\Omega}
    \coloneqq
    \operatorname{diag}
    (\bm B_{\partial\Omega}\bm 1)
    -
    \bm B_{\partial\Omega}.
\end{equation}

\subsection{Estimation of Rayleigh quotients}
\label{sec:estimation_of_rayleigh_quotients}
Let \(m \in \mathbb N\), \(\Omega = [-1, 1]^m\) and let \(\square = \{0,\ldots,n_1\}\times\cdots\times\{0,\ldots,n_m\} \subset \mathbb N_0^m\) an index set. We estimate the Rayleigh quotients 
\begin{equation}
    \mathcal R_1 = \sup_{v \in \Pi_\square \setminus \{0\}} \frac{\|v\|_{H^1(\Omega)}}{\|v\|_{L^2(\Omega)}},
    \qquad
    \mathcal R_2 = \sup_{v \in \Pi_\square^\perp \setminus \{0\}} \frac{\|v\|_{L^2(\Omega)}}{\|v\|_{H^1(\Omega)}},
\end{equation}
\paragraph{Estimation of $\mathcal R_1$:} 
We employ a Markov–type inequality for polynomials. In the one-dimensional case it is known that every polynomial \(p\) with \(\deg p \le n\) fulfills
\begin{equation}
    \label{eq:markov_inequality}
    \|p'\|_{L^2(-1,1)} \le C_n \|p\|_{L^2(-1,1)},
\end{equation}
where the smallest constant satisfies \(C_n \lesssim n^2\), see \cite{Aleksov2016}.
Let \(v \in \Pi_\square\) and set \(n \coloneqq \max_{\bm \alpha \in \square} \max_i \alpha_i \). Applying equation~\eqref{eq:markov_inequality} to the one-dimensional slices \(x_i \mapsto v(\bm x)\) and integrating over the remaining variables yields
\begin{equation}
    \|\partial_j v\|_{L^2(\Omega)} 
    \le 
    C_n \|v\|_{L^2(\Omega)}.
\end{equation}
Hence
\begin{equation}
    \|v\|_{H^1(\Omega)}^2
    =
    \|v\|_{L^2(\Omega)}^2 + \sum_{i=1}^m \|\partial_i v\|_{L^2(\Omega)}^2
    \le
    (1+m C_n^2)\|v\|_{L^2(\Omega)}^2.
\end{equation}
Taking the supremum over \(v \in \Pi_\square \setminus \{0\}\) gives
\begin{equation}
    \mathcal R_1 \lesssim n^2.
\end{equation}

\paragraph{Estimation of $R_2$:} Let \(\psi_n : H^1(-1,1) \to \Pi_n\) denote the \(H^1(-1,1)\)-orthogonal projection onto the space of polynomials of degree at most \(n\). It is known that for every \(v \in H^1(-1,1)\)
\begin{equation}
    \|v - \psi_n v\|_{L^2(-1,1)} \lesssim n^{-1} \|v'\|_{L^2(-1,1)},
\end{equation}
see \cite[Eq.~(5.4.23)]{Canuto2007}. 
Multidimensional approximation on Cartesian-product domains follows from tensor-product polynomial expansions~\cite{Canuto2007}. Consequently, for \(\Omega = (-1,1)^m\) there exists a projection \(\psi_\square: H^1(\Omega) \to \Pi_\square\) fulfilling
\begin{equation}
    \label{eq:h1_approximation}
    \|v - \psi_\square v\|_{L^2(\Omega)} 
    \lesssim 
    n^{-1} |v|_{H^1(\Omega)}
\end{equation}
Now let \(v \in \Pi_\square^\perp\), where \(\Pi_\square^\perp\) denotes the \(H^1\)-orthogonal complement. Since \(\psi_\square v = 0\), we obtain
\begin{equation}
    \|v\|_{L^2(\Omega)} 
    = 
    \|v - \psi_\square v\|_{L^2(\Omega)} 
    \lesssim 
    n^{-1} |v|_{H^1(\Omega)} .
\end{equation}
Taking the supremum over \(v \in \Pi_\square^\perp \setminus\{0\}\) yields the following upper bound
\begin{equation}
    \mathcal R_2 
    = 
    \sup_{v \in \Pi_\square^\perp \setminus\{0\}} \left(1 + \frac{|v|^2_{H^1(\Omega)}}{\|v\|^2_{L^2(\Omega)}}\right)^{-1/2} \lesssim n^{-1}.
\end{equation}

\subsection{Comparison of conditioning bounds}
\label{sec:comparison_of_conditioning_bounds}
Using the \(L^2\)-based formulation in Equation~\eqref{eq:l2_residual_loss_functional} for the Laplace boundary value problem, and applying the Markov-type inequality from Equation~\eqref{eq:markov_inequality} and Theorem~\ref{theo:lions_and_magenes}, we obtain
\begin{equation}
    n^{-2} \|u\|_{H^1(\Omega)}
    \lesssim
    \|\Delta u\|_{L^2(\Omega)} + \|\operatorname{tr}(u)\|_{L^2(\partial\Omega)} 
    \lesssim 
    n^2 \|u\|_{H^1(\Omega)},
\end{equation}
for all \(u \in V\). Next, combining the previously derived coercivity and continuity estimates for SINOS (Prop.~\ref{prop:coercivity} and Prop.~\ref{prop:continuity}) with the Rayleigh quotient estimates from Section~\ref{sec:estimation_of_rayleigh_quotients}, yields
\begin{equation}
    n^{-1} \|u\|_{H^{-1}(\Omega)} 
    \lesssim 
    |u|_{-1, \Pi_\square},
    \qquad 
    |\Delta u|_{-1, \Pi_\square} 
    \lesssim 
    n^2 \|u\|_{H^1(\Omega)}.
\end{equation}
Analogously, we specify for D-SINOS (Prop.~\ref{prop:dual_coercivity} and Prop.~\ref{prop:dual_continuity})
\begin{equation}
    n^{-2} \|u\|_{H^{-1}(\Omega)} \lesssim |u|_{-1, \Pi_\square^*},
    \qquad 
    |\Delta u|_{-1, \Pi_\square^*} \lesssim \|u\|_{H^1(\Omega)}.
\end{equation}
Consequently, D-SINOS mitigates stiffness effects for the Laplace boundary value problem better than the \(L^2\)-based formulation and also better than SINOS. Specifically, SINOS satisfies
\begin{equation}
    n^{-1} \|u\|_{H^1(\Omega)}
    \le
    |\Delta u|_{-1, \Pi_\square} + \|\operatorname{tr}(u)\|_{H^{1/2}(\partial\Omega)}
    \le
    n^2 \|u\|_{H^1(\Omega)},
\end{equation}
for all \(u \in V\), and D-SINOS satisfies
\begin{equation}
    n^{-2} \|u\|_{H^1(\Omega)}
    \lesssim
    |\Delta u|_{-1, \Pi_\square^*} + \|\operatorname{tr}(u)\|_{H^{1/2}(\partial\Omega)}
    \lesssim
    \|u\|_{H^1(\Omega)},
\end{equation}
for all \(u \in V\). 

Hence, by means of Theorem~\ref{theo:lions_and_magenes}, the conditioning for elliptic boundary value problems~\ref{def:bvp} can be summarized as
\vspace{-0.5em}
\begin{equation}
\label{eq:summary_conditioning}
    \kappa_{L^2} \lesssim n^4, 
    \qquad
    \kappa_{\text{SINOS}} \lesssim n^3,
    \qquad
    \kappa_{\text{D-SINOS}} \lesssim n^2.
\end{equation}

\section{Numerical experiments}
\label{sec:numerical_experiments}

Our numerical experiments benchmark SINOS and D-SINOS in three complementary settings.

First, we compare D-SINOS with SINOS at the operator level in Section~\ref{sec:random_sampling_test}.
Here, we investigate how effective the action of the Laplace operator is balanced in the \(H^1\)-norm, providing a direct numerical test of the preconditioning effect predicted by the theory.

Second, we apply SINOS and D-SINOS to a Poisson problem with asymmetric nonhomogeneous boundary data in Section~\ref{sec:poisson_equation}, serving as a canonical second-order linear elliptic boundary value problem that verifies the predicted behavior in the context of PINNs.

Third, we consider Navier-Stokes equations, testing if SINOS and D-SINOS are also effective beyond purely elliptic equations, in a case where the Laplace operator remains a relevant component.

\subsection{Random sampling test of SINOS and D-SINOS operators}
\label{sec:random_sampling_test}
We begin with an experiment in 2D based on randomly generated input samples to assess the behavior of SINOS and D-SINOS with respect to the \(H^1\)-norm. Let \(\bm \Delta \coloneqq \bm D^{(1)}\bm D^{(1)} + \bm D^{(2)}\bm D^{(2)}\) denote the discretized Laplace operator. Given a function \(u \in V\), we consider the ratios
\begin{equation}
    \label{eq:discrete_norm_ratios}
    \frac{
        \|\Delta u\|_{-1,V}
    }{
        \|u\|_{1,V}
    } 
    = 
    \frac{\sqrt{
        (\bm \Delta \bm u)^\top \bm W \bm J^{-1} (\bm \Delta \bm u)
    }}{\sqrt{
        \bm u^\top \bm W \bm J \bm u
    }}, 
    \qquad
    \frac{
        \|\Delta u\|_{-1,V^*}
    }{
        \|u\|_{1,V}
    }
    =
    \frac{
        \sqrt{(\bm \Delta \bm u)^\top \bm W \bm K^{-1} (\bm \Delta \bm u)
    }}{\sqrt{
        \bm u^\top \bm W \bm J \bm u
    }},
\end{equation}
where the corresponding quadratic forms are defined in \eqref{eq:discrete_quadratic_forms} and \eqref{eq:discrete_quadratic_forms_negative_norm}. We evaluate their behavior and their computational cost in Figure~\ref{fig:norm_ratios}.

\paragraph{Direct method:} To compute the ratios in~\eqref{eq:discrete_norm_ratios}, we evaluate SINOS and D-SINOS by solving 
\begin{equation}
    \label{eq:linear_systems}
    \bm J \bm x = \bm \Delta \bm u, \qquad \bm K \bm x = \bm \Delta \bm u,
\end{equation}
with direct linear algebra, implemented via \texttt{scipy.linalg.solve}~\cite{Virtanen2020}.

\paragraph{Pseudoinverse:} Alternatively, we compute the solutions using the precomputed pseudoinverse of \(\bm J\) and \(\bm K\), implemented via \texttt{scipy.linalg.pinv}~\cite{Virtanen2020}.

\paragraph{Conjugate-Gradient method:}
We solve the linear system \eqref{eq:linear_systems} using a conjugate gradient (CG) method with zero initialization and Jacobi preconditioning~\cite{Quarteroni2007}, i.e., the inverse of the diagonal entries of \(\bm J\) and \(\bm K\), respectively. A relative tolerance of \(10^{-12}\) and a maximum of five iterations are used, implemented via \texttt{scipy.sparse.linalg.cg}~\cite{Virtanen2020}.

\paragraph{Random sampling:} We generate random samples \(\bm u\) with random sparsity by first drawing \(p \sim \mathcal{U}(0,1)\), then activating each component of \(\bm u\) independently with probability \(p\), ensuring at least one nonzero entry. Active components are drawn from \(\mathcal{U}[0,1]\), and \(\bm u\) is normalized in the discrete \(H^1\)-norm. The Laplace operator on \(V \subset H^1\) is then applied to \(\bm u\).

\begin{figure}[htbp]
    \centering
    \begin{subfigure}[t]{0.48\linewidth}
        \centering
        \includegraphics[width=\linewidth]{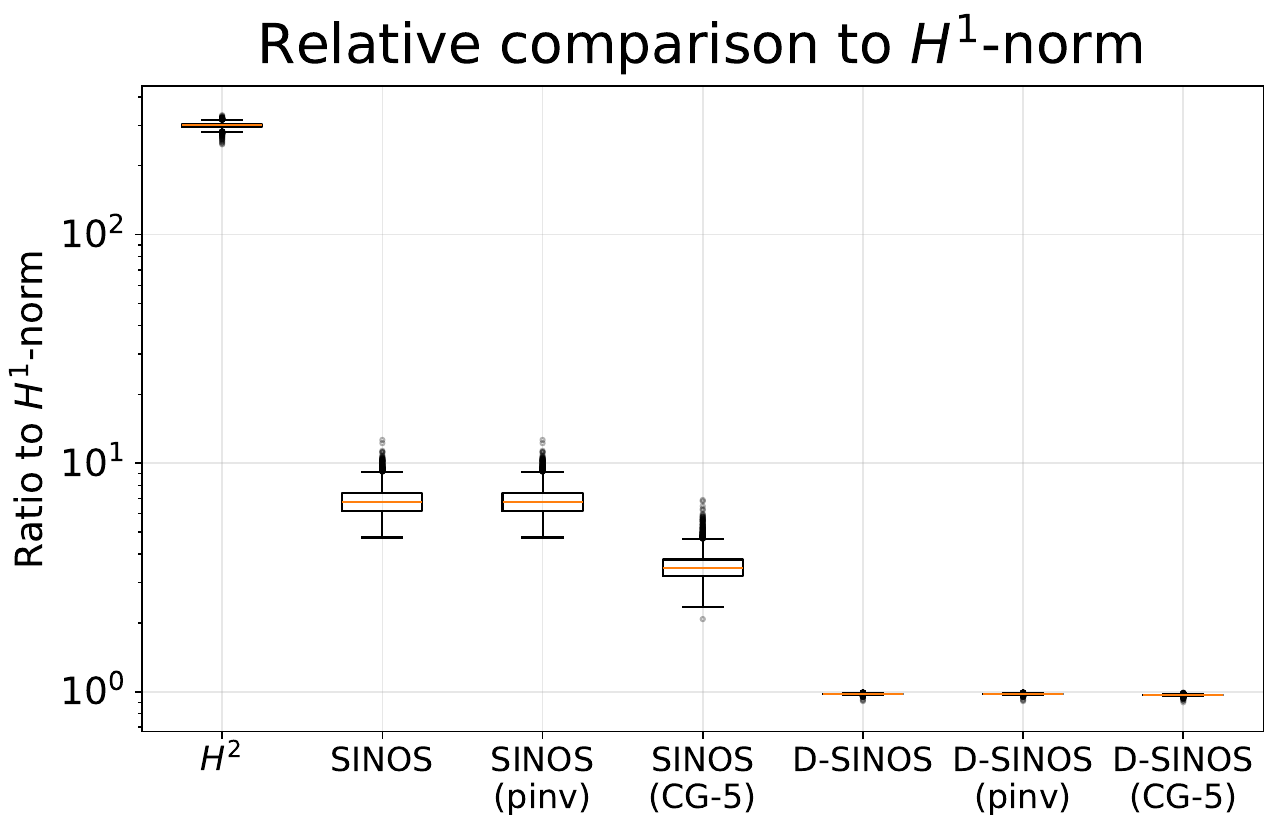}
        \caption{Values close to \(10^0\) indicate comparable magnitude to the \(H^1\)-norm.}
        \label{fig:norm_ratios_subfig1}
    \end{subfigure}
    \hfill
    \begin{subfigure}[t]{0.48\linewidth}
        \centering
        \includegraphics[width=\linewidth]{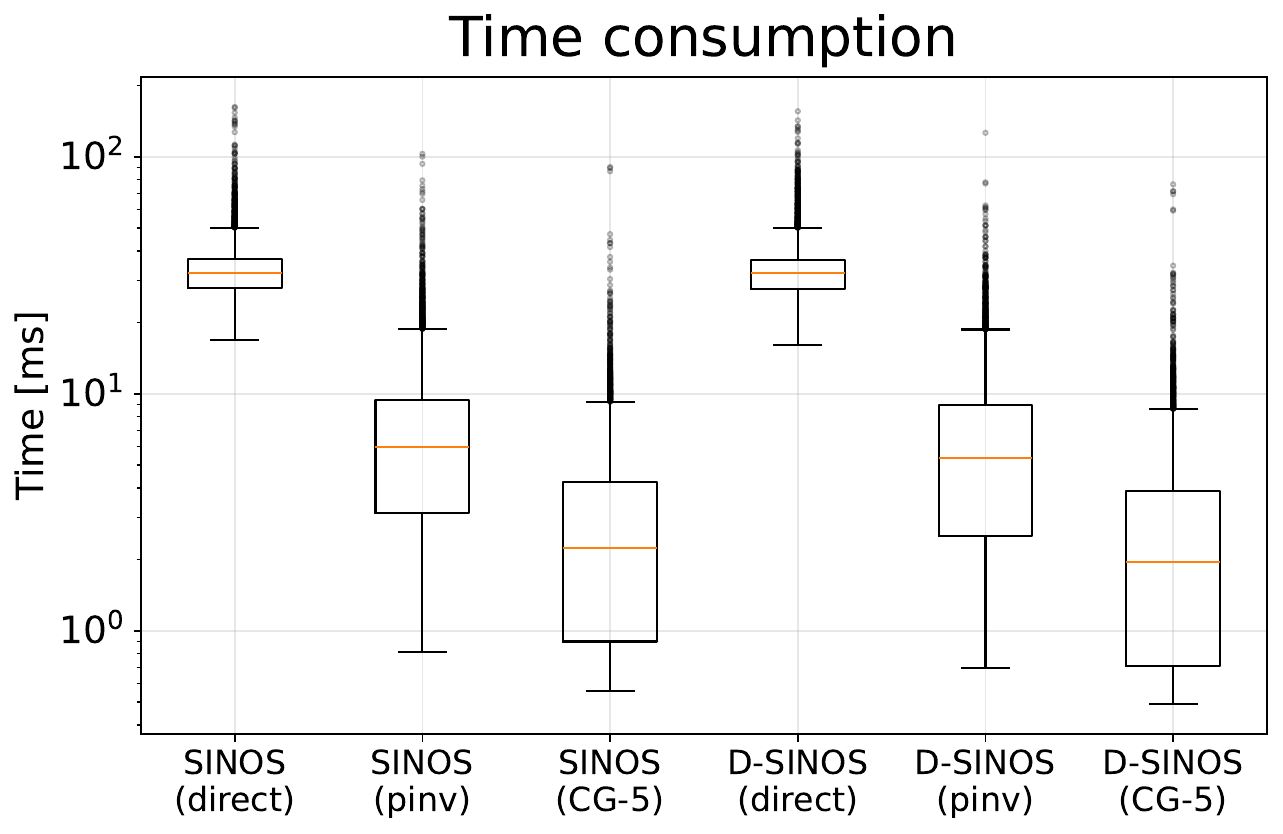}
        \caption{For each sample, the execution time is recorded in milliseconds.}
        \label{fig:norm_ratios_subfig2}
    \end{subfigure}
    
    \caption{Relative Comparison of SINOS, D-SINOS, and their CG-5 variants over 10{,}000 randomly generated samples on \((-1,1)^2\). The discretization uses tensorized Legendre-Lobatto nodes on a equidistant \(31 \times 31\) grid. 
    Each box-plot shows the sample-wise distribution using the standard Tukey convention, with whiskers extending to \(1.5\) times the IQR.}
    \label{fig:norm_ratios}
\end{figure}

\paragraph{Worst-case ratios:}
Using the experimental setup described in Figure~\ref{fig:norm_ratios}, we investigate the extremal behavior of the operators \(\bm J\) and \(\bm K\) by solving generalized eigenvalue problems. Given Proposition~\ref{prop:continuity} and Proposition~\ref{prop:dual_continuity}, we consider SINOS in the \(H^2\)-norm and D-SINOS in the \(H^1\)-norm, respectively.

We consider the symmetric matrix
\begin{equation}
    \bm B
    =
    \bm \Delta^\top \bm W \bm M^{-1} \bm \Delta,
\end{equation}
and solve the eigenproblem \(\bm B \bm x = \lambda \bm B_{H^k} \bm x\), with \(\bm B_{H^k}\) denoting the discrete \(H^k\)-inner product matrix 
\begin{equation}
    \|u\|_{H^k} = \sqrt{ \bm u^\top \bm B_{H^k} \bm u}
\end{equation} 
for \(u \in V\). We normalize the resulting dominant eigenvector with respect to the \(H^k\)-norm and use it to evaluate the maximum of the norm ratios in Table~\ref{tab:norm_ratios}. \\ 

\begin{table}[H]
    \centering
    \caption{Maximum norm ratios for SINOS and D-SINOS.}
    \label{tab:norm_ratios}
    
    \begin{tabular}{l|c|c|c|c}
    \textbf{Method} & $\bm M^{-1}$ & \textbf{Norm} & \textbf{Ratio} & \textbf{Max value} \\
    \hline
    SINOS 
    & $\bm J^{-1}$ 
    & $k=2$ 
    & ${\|\Delta u\|_{-1,V}}\,/\,{\|u\|_{H^2(\Omega)}}$ 
    & $\approx 1.23$ \\
    D-SINOS 
    & $\bm K^{-1}$ 
    & $k=1$ 
    & ${\|\Delta u\|_{-1,V^*}}\,/\,{\|u\|_{H^1(\Omega)}}$ 
    & $\approx 0.85$
    \end{tabular}
\end{table}

In conclusion, Figure~\ref{fig:norm_ratios_subfig1} clearly shows that D-SINOS resolves the stiffness problem
\begin{equation}
    \|\Delta u\|_{-1,V^*} \approx \mathcal O \left(\|u\|_{H^1(\Omega)}\right).
\end{equation} 
In contrast, SINOS does not fully resolve the stiffness, but partially mitigates it. Figure~\ref{fig:norm_ratios_subfig1} suggests that 
\begin{equation}
    \|\Delta u\|_{-1,V} \ll \|\Delta u\|_{L^2(\Omega)} =\mathcal{O}(\|u\|_{H^2(\Omega)}).
\end{equation} 
However, the worst-case analysis in Table~\ref{tab:norm_ratios} reveals that there exist instances in which the amplification for SINOS exceeds the \(H^2\)-norm, whereas for D-SINOS it remains below the corresponding \(H^1\)-norm, confirming that D-SINOS effectively controls the stiffness. Moreover, Figure~\ref{fig:norm_ratios_subfig2} shows that SINOS and D-SINOS can be computed efficiently using a CG-5 method while retaining its ability to resolve stiffness.

\subsection{Application to PINNs}
\label{sec:application_to_pinns}
Now, we confirm our theoretical findings for elliptic boundary value problems~\ref{def:bvp}.
We further demonstrate that D-SINOS remains effective beyond elliptic equations by applying it to the stationary incompressible Navier-Stokes equations.

SINOS and D-SINOS have already been applied to a various examples in~\cite{Suarez2024}, but they were combined with the \(L^2\) boundary loss. 
Here, we use a proper approximation of the \(H^{1/2}\) boundary loss and study the optimization behavior in detail. This means tracking the error during training, and comparing how fast the different formulations converge.

We also address the practical realization of the SINOS and D-SINOS operators by using a fixed Gauss-Legendre reference grid \(P_\Omega\) which decomposes the physical domain into replicas. 
In two dimensions, this corresponds to a decomposition into \(r_x \times r_y\) cells.
This yields sparse representations for SINOS and D-SINOS, and makes the approach suitable for PINN computations.

\subsubsection{Poisson equation with asymmetric nonzero boundary}
\label{sec:poisson_equation}
We consider the Poisson problem for \(\Omega = (-1,1)^2\) and \(u \in H^2(\Omega)\)
\begin{equation} 
    \label{eq:poisson-problem}
    \Delta u = f \quad \text{in } \Omega \coloneqq (-1,1)^2, 
    \qquad 
    u = g \quad \text{on } \partial\Omega , 
\end{equation} 
where we use the exact solution
\begin{equation}
    \label{eq:poisson_solution}
    u^*(x,y) 
    = 
    \alpha_{\exp} \exp(ax+by) + \alpha_{\mathrm{osc}} \sin(k_x x + k_y y) + \alpha_{\mathrm{poly}} xy + \ell_x x + \ell_y y,
\end{equation}
to prescribe the forcing term and boundary data as
\begin{equation}
    f = \Delta u^*, \qquad g = u^*_{\big|\partial\Omega}.
\end{equation}
The parameters are chosen as 
\begin{equation} 
    \begin{gathered} 
        \alpha_{\exp}=0.25, \qquad \alpha_{\mathrm{osc}}=0.75,\qquad \alpha_{\mathrm{poly}}=0.5,\\ \ell_x=0.10,\qquad \ell_y=-0.20,\qquad a=2.0,\qquad b=-0.75,\qquad k_x=8.0,\qquad k_y=2.0. 
    \end{gathered} 
\end{equation} 
This test problem combines anisotropic, oscillatory, polynomial, and linear components. The solution is neither axis-symmetric nor homogeneous on the boundary, making it well suited for testing the treatment of nontrivial boundary data.

We use a fully connected PINN with sine activation functions and five hidden layers of width 50. 
The Poisson problem is discretized on a tensor-product Gauss-Legendre grid with polynomial degrees \newline \(n_x = n_y = 10\). 
The domain is decomposed into \(r_x \times r_y = 5 \times 5\) cells, resulting in
\begin{equation}
    2r_y(n_y+1) + 2r_x(n_x+1)
    =
    220
\end{equation}
boundary collocation points. We compare the training of the three loss functions of the form
\begin{equation}
    \mathcal L_{\Omega} + \mathcal L_{\partial \Omega} \longrightarrow \min,
\end{equation}
detailed in Table~\ref{tab:losses_poisson}. \\

\begin{table}[htbp]
\centering
\caption{Loss functions used in the Poisson experiment.}
\label{tab:losses_poisson}
    \begin{tabular}{lll}
    \toprule
    Method & Domain loss $\mathcal L_\Omega$ & Boundary loss $\mathcal L_{\partial\Omega}$ \\
    \midrule
    MSE
    &
    \(
    \frac{1}{|P_\Omega|}
    \bm r_\Omega^\top \bm r_\Omega
    \)
    &
    \(
    \frac{1}{|P_{\partial\Omega}|}
    \bm r_{\partial\Omega}^\top \bm r_{\partial\Omega}
    \)
    \\[2ex]
    
    SINOS
    &
    \(
        \bm r_\Omega^\top \bm W \bm J^{-1} \bm r_\Omega
    \)
    &
    \(
        \bm r_{\partial\Omega}^{\top} \bm W_{1/2, \partial\Omega} \bm r_{\partial\Omega}
    \)
    \\[2ex]
    
    D-SINOS
    &
    \(
        \bm r_\Omega^{\top} \bm W \bm K^{-1} \bm r_\Omega
    \)
    &
    \(
        \bm r_{\partial\Omega}^{\!\top} \bm W_{1/2,\partial\Omega} \bm r_{\partial\Omega}
    \)
    \\
    \bottomrule
    \end{tabular}
\end{table}

Here, \(\bm r_\Omega\) and \(\bm r_{\partial\Omega}\) denote the PDE and boundary residual vectors, respectively.

Figure~\ref{fig:ground_truth_poisson} shows the ground truth \(u^*\). 
The PDE-to-boundary loss ratios in Figure~\ref{fig:loss_ratio_poisson} compare the evolution of the PDE and its boundary during L-BFGS training~\cite{Liu1989}. Figures~\ref{fig:rmse_history_poisson} and~\ref{fig:max_error_history_poisson} show the corresponding RMSE and maximum error histories on the evaluation grid. Figure~\ref{fig:results_poisson} compares the trained predictions and the pointwise absolute errors. \\

\begin{figure}[htbp]
    \centering
    \begin{subfigure}[t]{.48\textwidth}
        \centering
        \includegraphics[width=0.85\textwidth]{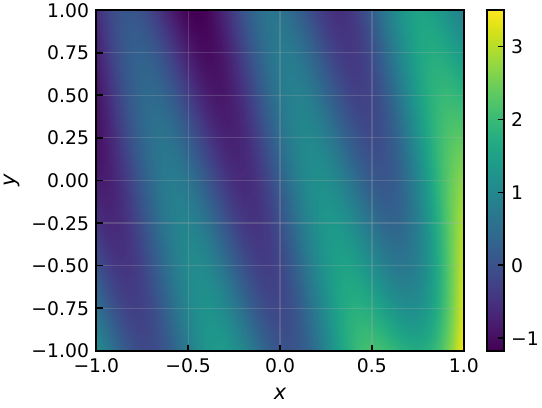}
        \caption{Ground truth \(u^*\).}
        \label{fig:ground_truth_poisson}
    \end{subfigure}
    % \hfill
    \begin{subfigure}[t]{.48\textwidth}
        \centering
        \includegraphics[width=0.8\textwidth]{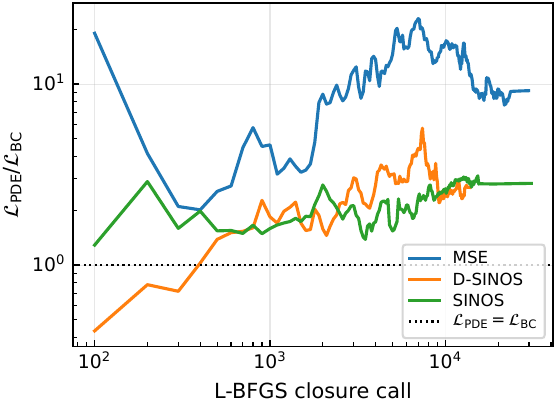}
        \caption{PDE-to-boundary loss ratio history.}
        \label{fig:loss_ratio_poisson}
    \end{subfigure}
    
    \vspace{1.25em}
    
    \begin{subfigure}[t]{.48\textwidth}
        \centering
        \includegraphics[width=0.8\textwidth]{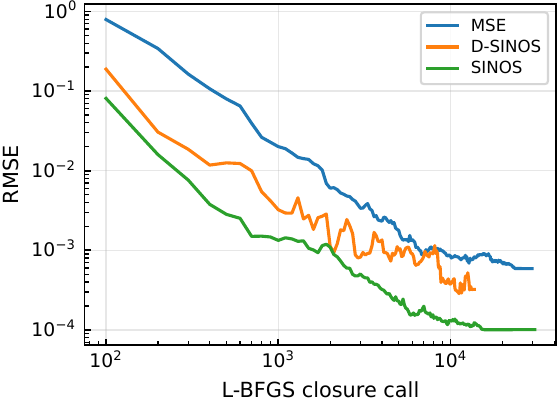}
        \caption{RMSE history.}
        \label{fig:rmse_history_poisson}
    \end{subfigure}
    \hfill
    \begin{subfigure}[t]{.48\textwidth}
        \centering
        \includegraphics[width=0.8\textwidth]{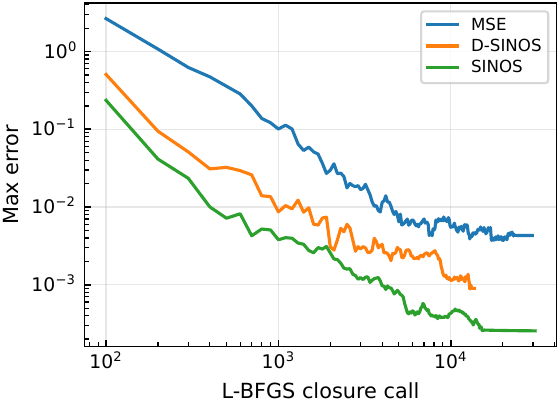}
        \caption{Maximum-error history.}
        \label{fig:max_error_history_poisson}
    \end{subfigure}
    \caption{
        Ground truth, training loss histories, and RMSE and maximum error histories for MSE, SINOS,
        and D-SINOS of the Poisson PINN experiment are shown. For the evaluation, an equidistant
        \(201 \times 201\) grid is used.
    }
    \label{fig:overview_poisson}
\end{figure}

\begin{figure}[htbp] 
    \centering
    \includegraphics[width=0.85\textwidth]{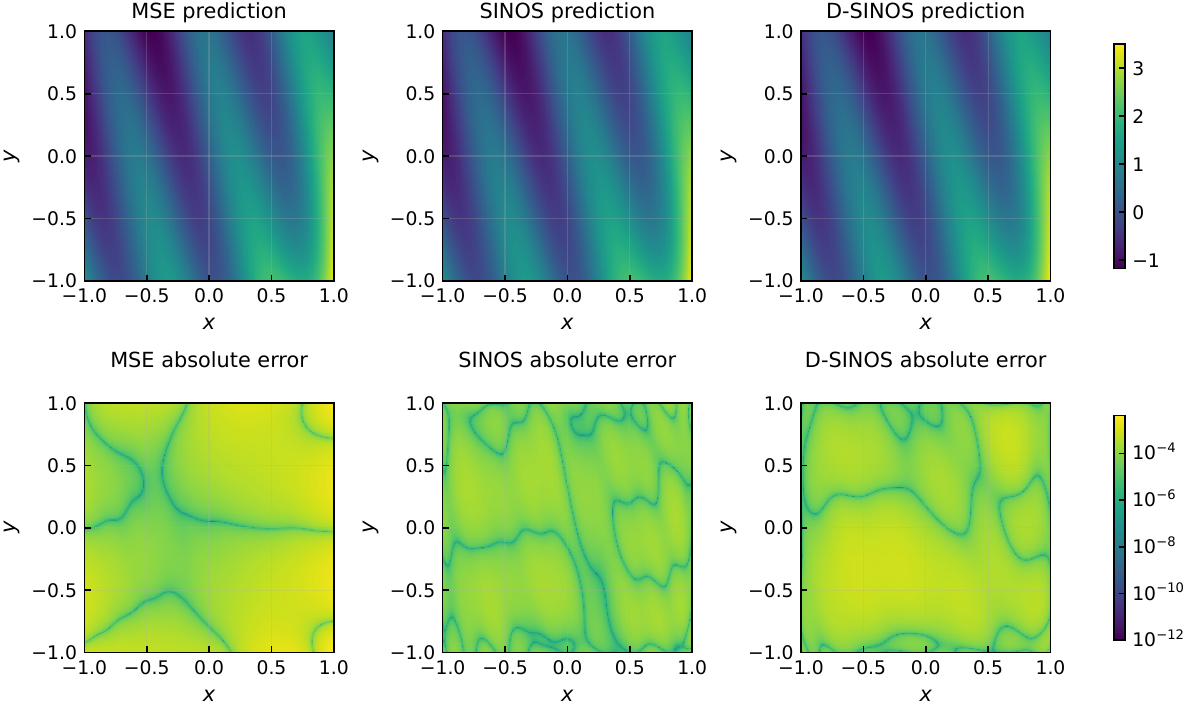} 
    \caption{
        Final predictions and pointwise absolute errors for MSE, SINOS, and D-SINOS. 
        For the evaluation, an equidistant \(201 \times 201\) grid is used.
    } 
    \label{fig:results_poisson}
\end{figure}

\subsubsection{Stationary incompressible Navier-Stokes equations}
We consider the stationary incompressible Navier-Stokes problem for \(\Omega = (-1,1)^2\) 
\begin{equation} 
    \begin{aligned} 
        -\nu \Delta u + (u \cdot \nabla)u + \nabla p &= f && \text{in } \Omega, \\ 
        \nabla \cdot u &= 0 && \text{in } \Omega, \\ 
        u &= g && \text{on } \partial \Omega . 
    \end{aligned}
\end{equation}
Here, \(u \in H^2(\Omega)^2\) denotes the velocity field satisfying the prescribed Dirichlet boundary condition, \(p \in H^1(\Omega)\) the pressure, \(f\) an external force, and \(\nu>0\) the viscosity. 

The term \(-\nu \Delta u\) is the second-order elliptic diffusion part, while \((u \cdot \nabla)u\) is a nonlinear first-order transport term. The pressure gradient \(\nabla p\) couples velocity and pressure, and \(\nabla \cdot u = 0\) imposes incompressibility. Hence, this test problem is a nonlinear coupled system with an elliptic diffusion component, making it a suitable test case for further assessing SINOS and D-SINOS.

As in the previous test problem, we use an exact solution and choose the boundary data and forcing term accordingly. Let the stream function
\begin{equation}
    \psi(x,y) = A \sin(\pi x)\sin(\pi y),
\end{equation} 
and define the exact velocity field
\begin{equation} 
    u^*(x,y) 
    = 
    \begin{pmatrix} 
        \partial_y \psi(x,y) \\ 
        -\partial_x \psi(x,y) 
    \end{pmatrix} 
    = 
    A\pi \begin{pmatrix} 
        \sin(\pi x)\cos(\pi y) \\ 
        -\cos(\pi x)\sin(\pi y) 
    \end{pmatrix}. 
\end{equation} 
This construction implies \(\nabla \cdot u^* = 0\). 
The pressure is chosen as 
\begin{equation} 
    p^*(x,y) = \alpha_p \sin(2\pi x + \pi y), 
\end{equation} 
where \(A\) and \(\alpha_p\) are fixed. 
The Dirichlet boundary data are given by the trace of the velocity
\(
    g = u^*|_{\partial \Omega}. 
\)
The force term is then obtained by inserting \((u^*,p^*)\) into the equation
\begin{equation} 
    f = -\nu \Delta u^* + (u^* \cdot \nabla)u^* + \nabla p^*.
\end{equation} 
Thus, 
\(
    f = (f_1, f_2)^\top 
\) 
is given by 
\begin{align} 
    f_1(x,y) 
    &= 
    2\nu \pi^2 u^\ast_1(x,y) + A^2\pi^3 \sin(\pi x)\cos(\pi x) + 2\pi \alpha_p \cos(2\pi x+\pi y), \\ 
    f_2(x,y) 
    &= 
    2\nu \pi^2 u^\ast_2(x,y) + A^2\pi^3 \sin(\pi y)\cos(\pi y) + \pi \alpha_p \cos(2\pi x+\pi y). 
\end{align}
The parameters are chosen as
\begin{equation}
    \nu = 0.05, 
    \qquad 
    A = 1.0, 
    \qquad 
    \alpha_p = 0.5.
\end{equation}

We use a fully connected PINN with sine activation functions and five hidden layers of width 80. 
The Navier-Stokes problem is discretized on a tensor-product Gauss-Legendre grid with polynomial degrees \(n_x = n_y = 10\). 
The domain is decomposed into \(r_x \times r_y = 8 \times 8\) cells, resulting in
\begin{equation}
    2r_y(n_y+1) + 2r_x(n_x+1)
    =
    352
\end{equation}
boundary collocation points. We compare the training of the three loss functions of the form
\begin{equation}
    \mathcal L_{\text{mom}} + \mathcal L_{\text{div}} + \mathcal L_{\partial \Omega}  \longrightarrow \min,
\end{equation}
detailed in Table~\ref{tab:losses_nse}. \\

\begin{table}[htbp]
    \centering
    \caption{Loss functions used in the stationary incompressible Navier-Stokes experiment.}
    \label{tab:losses_nse}
    \small
    \begin{tabular}{llll}
    \toprule
    Method 
    & Momentum loss $\mathcal L_{\mathrm{mom}}$ 
    & Incompressibility loss $\mathcal L_{\mathrm{div}}$ 
    & Boundary loss $\mathcal L_{\partial\Omega}$ \\
    \midrule
    
    MSE
    &
    \(
    \frac{1}{|P_\Omega|}
    \bm r_{\mathrm{mom}}^\top \bm r_{\mathrm{mom}}
    \)
    &
    \(
    \frac{1}{|P_\Omega|}
    \bm r_{\mathrm{div}}^\top \bm r_{\mathrm{div}}
    \)
    &
    \(
    \frac{1}{|P_{\partial\Omega}|}
    \bm r_{\partial\Omega}^\top \bm r_{\partial\Omega}
    +
    \left(
        \frac{1}{|P_\Omega|}
        \bm 1^\top \bm p
    \right)^2
    \)
    \\[2ex]
    
    SINOS
    &
    \(
    \bm r_{\mathrm{mom}}^\top \bm W \bm J^{-1} \bm r_{\mathrm{mom}}
    \)
    &
    \(
    \bm r_{\mathrm{div}}^\top \bm W \bm r_{\mathrm{div}}
    \)
    &
    \(
    \bm r_{\partial\Omega}^{\top} \bm W_{1/2,\partial\Omega} \bm r_{\partial\Omega}
    +
    \left(
        \bm 1^\top \bm W \bm p
    \right)^2
    \)
    \\[2ex]
    
    D-SINOS
    &
    \(
    \bm r_{\mathrm{mom}}^\top \bm W \bm K^{-1} \bm r_{\mathrm{mom}}
    \)
    &
    \(
    \bm r_{\mathrm{div}}^\top \bm W \bm r_{\mathrm{div}}
    \)
    &
    \(
    \bm r_{\partial\Omega}^{\top} \bm W_{1/2, \partial\Omega} \bm r_{\partial\Omega}
    +
    \left(
        \bm 1^\top \bm W \bm p
    \right)^2
    \)
    \\
    \bottomrule
    \end{tabular}
\end{table}

Here, \(\bm r_{\mathrm{mom}}\) denotes the residual vector of the two momentum equations, 
\(\bm r_{\mathrm{div}}\) the incompressibility residual, \(\bm r_{\partial\Omega}\) the velocity boundary residual, 
and \(\bm p\) the vector of predicted pressure values.
The pressure gauge loss fixes the additive constant of the pressure. 
For SINOS and D-SINOS, it is evaluated using the same quadrature rule as the interior loss.

The ground truth velocity and pressure are shown in Figure~\ref{fig:ground_truth_nse}.
The PDE-to-boundary loss ratios in Figure~\ref{fig:loss_ratio_nse} show the evolution of the PDE and its boundary during L-BFGS training~\cite{Liu1989}.
Figures~\ref{fig:rmse_history_velocity_nse}, \ref{fig:max_error_history_velocity_nse}, \ref{fig:rmse_history_pressure_nse}, and~\ref{fig:max_error_history_pressure_nse} show the corresponding RMSE and maximum-error histories for the velocity and pressure.
Figures~\ref{fig:results_velocity_nse} and~\ref{fig:results_pressure_nse} compare the trained velocity and pressure predictions with their pointwise absolute errors. \\

\begin{figure}[htbp] 
    \centering 
    \includegraphics[width=0.85\textwidth]{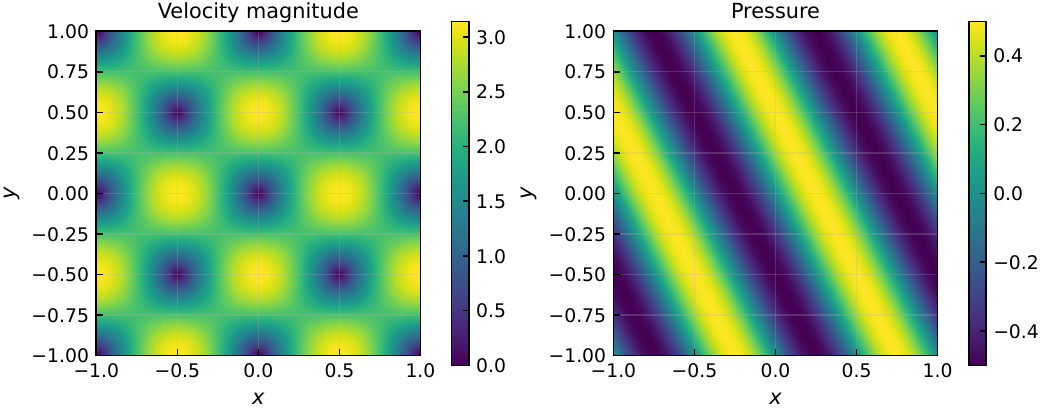} 
    \caption{Ground truth velocity \(u^*\) and pressure \(p^*\) for the stationary incompressible Navier-Stokes experiment.} 
    \label{fig:ground_truth_nse}
\end{figure}

\begin{figure}[htbp]
    \centering

    \begin{subfigure}[t]{0.5\textwidth}
        \centering
        \includegraphics[width=\textwidth]{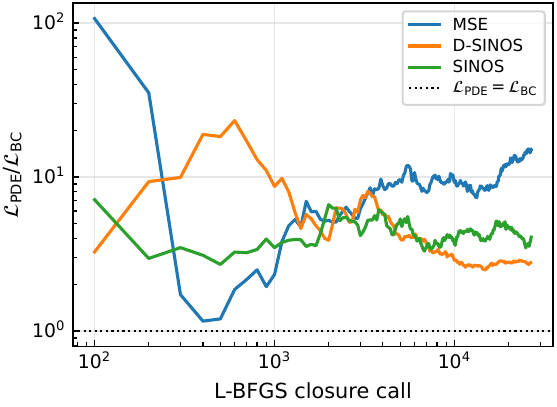}
        \caption{PDE-to-boundary loss ratio history.}
        \label{fig:loss_ratio_nse}
    \end{subfigure}

    \vspace{0.8em}

    \begin{subfigure}[t]{0.48\textwidth}
        \centering
        \includegraphics[width=\textwidth]{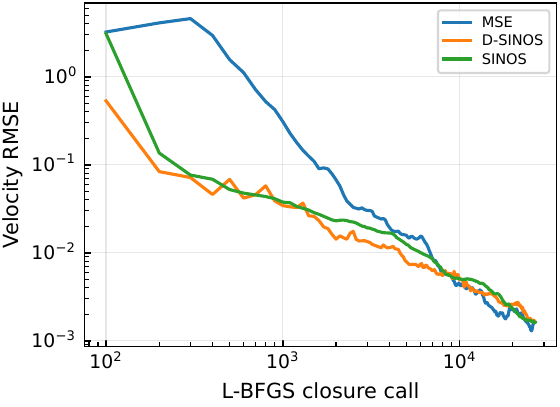}
        \caption{Velocity RMSE history.}
        \label{fig:rmse_history_velocity_nse}
    \end{subfigure}
    \hfill
    \begin{subfigure}[t]{0.48\textwidth}
        \centering
        \includegraphics[width=\textwidth]{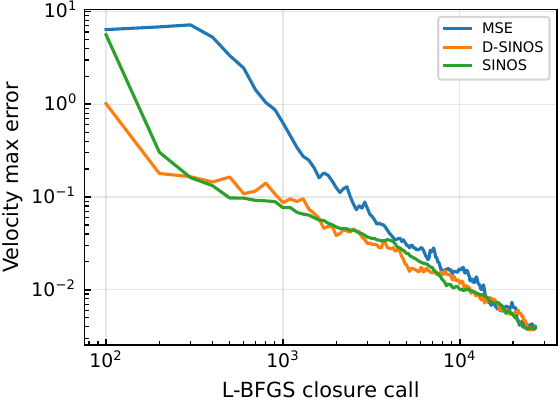}
        \caption{Velocity maximum-error history.}
        \label{fig:max_error_history_velocity_nse}
    \end{subfigure}

    \vspace{0.6em}

    \begin{subfigure}[t]{0.48\textwidth}
        \centering
        \includegraphics[width=\textwidth]{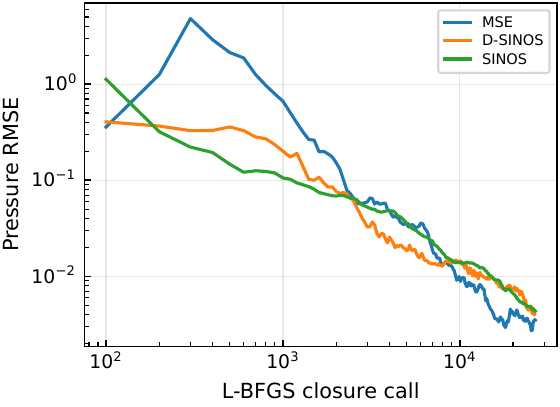}
        \caption{Pressure RMSE history.}
        \label{fig:rmse_history_pressure_nse}
    \end{subfigure}
    \hfill
    \begin{subfigure}[t]{0.48\textwidth}
        \centering
        \includegraphics[width=\textwidth]{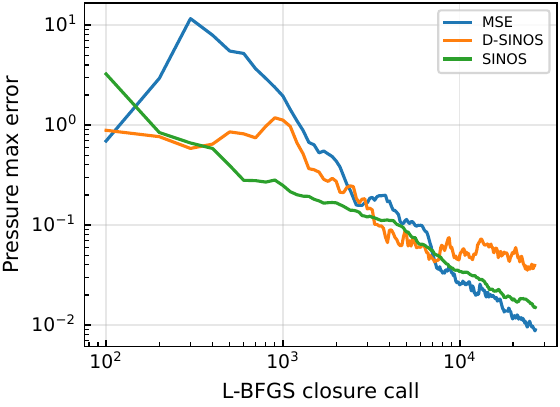}
        \caption{Pressure maximum-error history.}
        \label{fig:max_error_history_pressure_nse}
    \end{subfigure}
    \caption{
        Training loss histories, RMSE and maximum-error histories for MSE, SINOS,
        and D-SINOS of the Navier-Stokes PINN experiment are shown. For the evaluation, an equidistant \(201 \times 201\) grid is used.
    }
    \label{fig:overview_nse}
\end{figure}

\begin{figure}[htbp] 
    \centering 
    \includegraphics[width=0.95\textwidth]{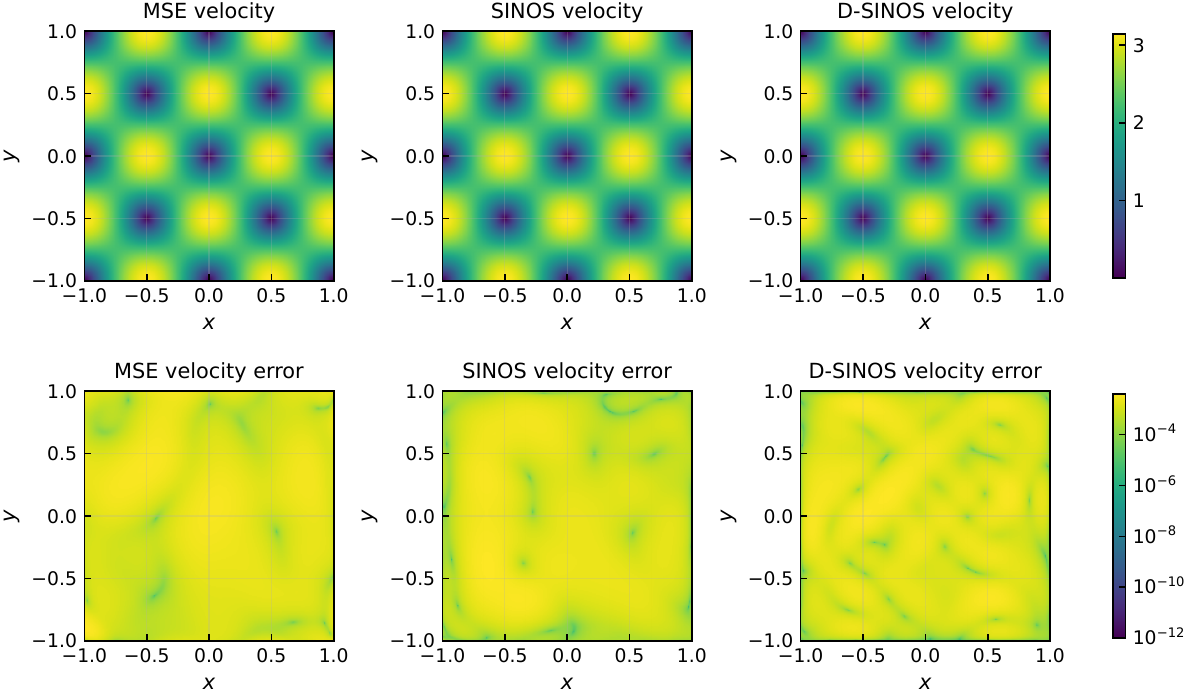}
    \caption{
        Velocity predictions and pointwise absolute velocity errors for MSE, SINOS, and D-SINOS after \(2000\) L-BFGS iterations.
        For the evaluation, an equidistant \(201 \times 201\) grid is used.
    } 
    \label{fig:results_velocity_nse} 
\end{figure}

\begin{figure}[htbp]
    \centering
    \includegraphics[width=0.95\textwidth]{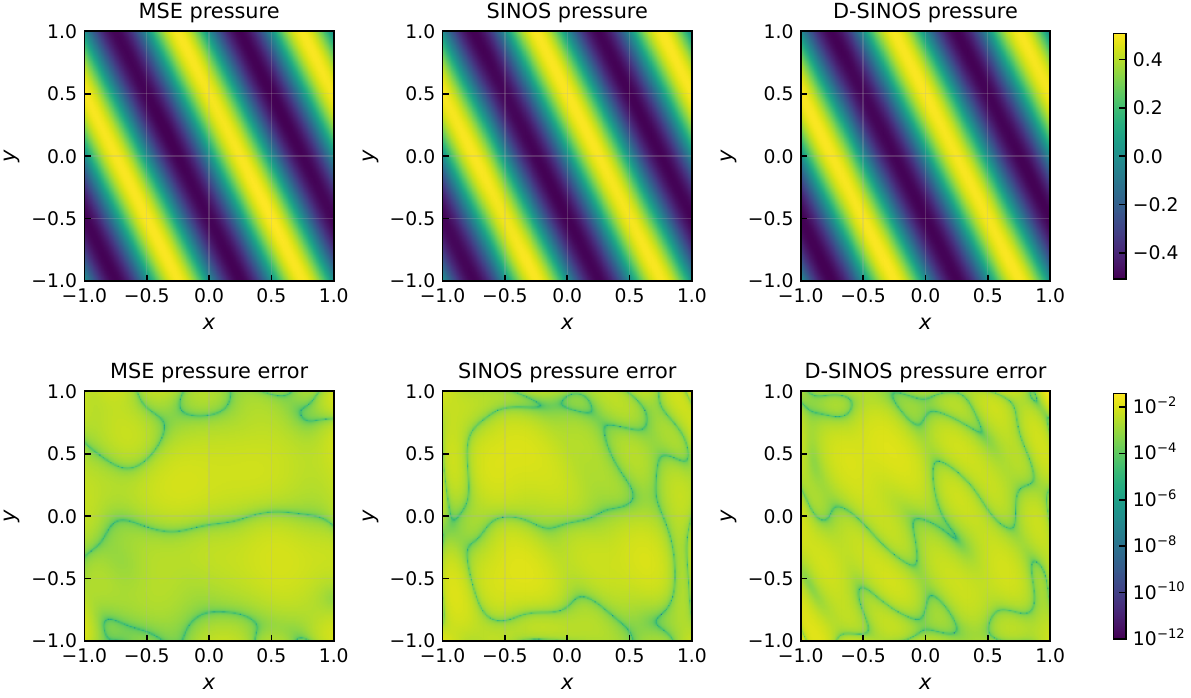}
    \caption{
        Pressure predictions and pointwise absolute pressure errors for MSE, SINOS, and D-SINOS after \(2000\) L-BFGS iterations.
        For the evaluation, an equidistant \(201 \times 201\) grid is used.
    }
    \label{fig:results_pressure_nse}
\end{figure}

\subsubsection{Summary of PINN results}

\paragraph{Poisson experiment:}

The PDE-to-boundary loss ratio in Figure~\ref{fig:loss_ratio_poisson} shows that SINOS and D-SINOS reduce the imbalance between the PDE and boundary residuals. In contrast to the MSE formulation, their ratios remain closer to one over the course of the L-BFGS optimization. This indicates a more balanced optimization problem and is consistent with the predicted mitigation of stiffness effects.

The error histories in Figure~\ref{fig:rmse_history_poisson} and Figure~\ref{fig:max_error_history_poisson} show that this improved balance affects the convergence behavior and final prediction quality. 
Among the three methods, SINOS performs best, D-SINOS shows intermediate performance, and the MSE baseline performs worst.
For the RMSE, SINOS approaches an error of \(10^{-5}\), D-SINOS reaches \(10^{-4}\), while the MSE baseline only reaches \(10^{-3}\).
A similar behavior is observed for the maximum error, where SINOS reaches an error of order \(10^{-4}\), while D-SINOS only just reaches the \(10^{-4}\) level and the MSE baseline remains in the upper \(10^{-3}\) range, closer to \(10^{-2}\).

Figure~\ref{fig:results_poisson} shows that all three methods produce visually accurate predictions. 
The corresponding error plots reveal localized regions where the error is very small, while residual errors remain distributed across the domain. 
SINOS yields the smallest errors overall, mostly around \(10^{-4}\) or below, whereas MSE and D-SINOS retain larger error levels in parts of the domain.

Overall, this experiment confirms that the proposed loss formulations mitigate stiffness in practice, with SINOS giving the best performance for the Poisson test case.

\vspace{-0.75em}

\paragraph{Navier-Stokes experiment:}
This test case is more challenging than the Poisson problem, since the elliptic contribution enters only through the viscous term \(-\nu \Delta u\). With the small viscosity \(\nu=0.05\), this contribution is relatively weak. 

The error histories in Figures~\ref{fig:rmse_history_velocity_nse}--\ref{fig:max_error_history_pressure_nse} show that SINOS and D-SINOS produce a rapid error decay during the first L-BFGS closure calls.

For the velocity, all three methods reach comparable final accuracies.
The final velocity RMSE and the final maximum velocity error are of order \(10^{-3}\).
For the pressure, all three methods reach comparable RMSE values of order \(10^{-3}\).
For the maximum pressure error, SINOS and D-SINOS only reach errors of order \(10^{-2}\), while D-SINOS stagnates at a larger error level. 
The MSE baseline reaches a maximum pressure error of order \(10^{-3}\).

The prediction plots in Figures~\ref{fig:results_velocity_nse} and~\ref{fig:results_pressure_nse} show that all three methods reproduce the qualitative structure of the velocity magnitude and pressure field.
The corresponding absolute error plots reveal small localized regions with very low error, while residual errors remain distributed across the domain. 
For the velocity field, the errors are mostly of order \(10^{-3}\) to \(10^{-4}\), whereas the pressure errors are generally larger with an order of around \(10^{-2}\).

Overall, this experiment shows that SINOS improves the optimization behavior also when the elliptic contribution is weak. 
However, the pressure results show that improved residual balancing does not always translate into uniformly better prediction errors for all components of the solution.
SINOS is the more robust choice in this experiment, since it avoids the pressure error stagnation observed for D-SINOS while retaining the improved early convergence behavior.

\section{Conclusion}

We studied numerical stiffness in least-squares formulations of PINNs for elliptic boundary value problems with arbitrary nonzero Dirichlet boundary data. 
Starting from the observation that standard \(L^2\)-based PINN losses overemphasize the PDE residual and underweight the boundary residual, we reformulated the residual loss in negative Sobolev norms. 
In particular, we used the \(H^{-1}\) norm for the PDE residual and the \(H^{1/2}\) norm for the boundary residual, which yields a well-posed residual functional, equivalent to the squared \(H^1\) error, see Theorem~\ref{theo:lions_and_magenes} and Corollary~\ref{cor:residual_loss_functional}.

Then, we defined and analyzed SINOS and D-SINOS, two finite-dimensional approximations of the \(H^{-1}\) norm. 
SINOS restricts the \(H^{-1}\) test space to a finite-dimensional subspace, while D-SINOS additionally uses a norm based on adjoint derivative operators.
We proved continuity and coercivity estimates for both and showed that the constants are governed by Rayleigh quotients of the approximation space, see Propositions~\ref{prop:continuity}-\ref{prop:dual_coercivity}.
For polynomial spaces, these estimates yield improved conditioning compared with the standard \(L^2\) loss. 
The \(L^2\)-based formulation scales as \(\kappa_{L^2} \in \mathcal O(n^4)\), SINOS as \(\kappa_{\text{SINOS}} \in \mathcal O(n^3)\), and D-SINOS as \(\kappa_{\text{D-SINOS}} \in \mathcal O(n^2)\), see Section~\ref{sec:estimation_of_rayleigh_quotients} and Section~\ref{sec:comparison_of_conditioning_bounds}.

We discretized these formulations using Sobolev cubature rules in polynomial spaces, see Section~\ref{sec:polynomial_space} and Section~\ref{sec:sobolev_cubature_rules}.
The resulting loss terms can be evaluated either by precomputed pseudoinverses or by iterative solvers such as conjugate gradients.
The operator-level experiments confirmed the predicted rebalancing and conditioning behavior, see Section~\ref{sec:random_sampling_test}, while the PINN experiments for the Poisson and stationary incompressible Navier-Stokes equations showed that the theoretical improvements translate into better training behavior and more accurate solutions, see Section~\ref{sec:application_to_pinns}.
Overall, the experiments support the main theoretical conclusion that negative Sobolev residual losses, in practice especially the SINOS discretization, mitigates stiffness and boosts convergence for elliptic PDEs.

\paragraph{\small Acknowledgments}
{\small
    We acknowledge 
    Juan-Esteban Suarez for insightful comments and helpful discussions.
}
\vspace{-1em}
\paragraph{\small Funding}
{\small
    This work was partially funded by the Center for Advanced Systems Understanding (CASUS), financed by Germany's Federal Ministry of Education and Research (BMBF) and by the Saxon Ministry for Science, Culture and Tourism (SMWK) with tax funds on the basis of the budget approved by the Saxon State Parliament.
}
\vspace{-1em}
\paragraph{\small Code availability}
{\small 
    The source code is available at \url{https://github.com/phil-hofmann/SINOS}.
}
% \vspace{-1em}
% \paragraph{\small Competing interests}
% {\small
%     The authors declare no competing interests.
% }
% \vspace{-1em}
% \paragraph{\small Ethics approval and consent to participate}
% {\small
%     Not applicable.
% }

\newpage

\appendix

\section{Proof of Theorem \ref{theo:negative_norm_decomposition}}
\label{app:negative_norm_decomposition}
Since \(H^1(\Omega) = V \oplus V^\perp\), every \(u \in H^1(\Omega)\) decomposes uniquely as \(u = v + w\) with \(v \in V\) and \(w \in V^\perp\). Setting \(\alpha \coloneqq |u|_{-1, V}\), \(\beta \coloneqq |u|_{-1, V^\perp }\), we obtain the upper bound
\begin{equation}
    \langle u, v + w \rangle_{L^2(\Omega)} 
    \le
    \alpha \|v\|_{H^1(\Omega)} + \beta \|w\|_{H^1(\Omega)}.
\end{equation}
Using the Pythagorean identity in \(H^1(\Omega)\) in the first step and the Cauchy-Schwarz inequality for the euclidean inner product in the second step, we obtain
\begin{equation}
    \frac{\langle u, v + w\rangle_{L^2(\Omega)}}{\|v + w\|_{H^1(\Omega)}} 
    \le
    \frac{\alpha \|v\|_{H^1(\Omega)} + \beta \|w\|_{H^1(\Omega)}}{\sqrt{\|v\|_{H^1(\Omega)}^2 + \|w\|_{H^1(\Omega)}^2}}
    \le \sqrt{\alpha^2 + \beta^2}.
\end{equation}
Taking the supremum over all such \(v + w\), yields
\begin{equation}
    \|u\|_{H^{-1}(\Omega)} 
    \le
    \sqrt{\alpha^2 + \beta^2}.
\end{equation}
By definition \(\alpha\) and \(\beta\) are suprema. Hence, there exist sequences \(v_n \in V\) and \(w_n \in V^\perp\) with \(\|v_n\|_{H^1(\Omega)} = \|w_n\|_{H^1(\Omega)} = 1\) fulfilling
\begin{equation}
    \langle u, v_n \rangle_{L^2(\Omega)} \overset{n \rightarrow \infty}{\longrightarrow} \alpha, 
    \quad 
    \langle u, w_n \rangle_{L^2(\Omega)} \overset{n \rightarrow \infty}{\longrightarrow} \beta.
\end{equation}
If \(\alpha = \beta = 0\) it follows that \(\|u\|_{H^{-1}(\Omega)} = 0\) and the lower bound is trivial. Thus, assume that \(\alpha \neq 0\) or \(\beta \neq 0\) and set \(\tilde \alpha \coloneqq \alpha / \sqrt{\alpha^2 + \beta^2}\) and \(\tilde \beta \coloneqq \beta / \sqrt{\alpha^2 + \beta^2}\). Specifically, \(\tilde \alpha ^2 + \tilde \beta^2 = 1\). Now, set \(\tilde  v_n \coloneqq \tilde \alpha v_n \in U\) and \(\tilde w_n \coloneqq \tilde \beta w_n \in U^\perp\), then
\begin{equation}
    \|\tilde v_n + \tilde w_n\|_{H^1(\Omega)}^2 = \sqrt{\tilde \alpha^2 \|v_n\|_{H^1(\Omega)}^2 + \tilde \beta^2 \|w_n\|_{H^1(\Omega)}^2} = 1.
\end{equation}
Therefore,
\begin{equation}
    \langle u, \tilde v_n + \tilde w_n \rangle_{L^2(\Omega)} 
    = 
    \tilde \alpha \langle u, v_n \rangle_{L^2(\Omega)} + \tilde \beta \langle u, w_n \rangle_{L^2(\Omega)} 
    \overset{n \rightarrow \infty}{\longrightarrow} 
    \sqrt{\alpha^2 + \beta^2},
\end{equation}
which yields 
\begin{equation}
    \|u\|_{H^{-1}(\Omega)} \geq \lim_{n \rightarrow \infty} \langle u, \tilde v_n + \tilde w_n \rangle_{L^2(\Omega)} = \sqrt{\alpha^2 + \beta^2}.
\end{equation}
Finally, combining the upper and lower bound gives the asserted identity.
\hfill $\square$

\section{Proof of Theorem \ref{theo:infimum_representation}}
\label{app:infimum_representation}
Let \(u \in V\). For the proof, we denote the right hand side of \eqref{eq:3} by \(|u|_{\inf}\). The set of decompositions of \(u\) required by \(|u|_{\inf}\) is non-empty, since the trivial decomposition \(v_0 = u\) and \(v_1 = \ldots = v_m = 0\) is always valid. Hence, given such a decomposition \(v_0, v_1, \ldots, v_m \in V\) of \(u\), we find that for any \(w \in V\)
\begin{align}
    \langle u, w\rangle_{L^2(\Omega)} 
    & =
    \langle v_0 + \sum_{i=1}^m \partial_i v_i , w\rangle_{L^2(\Omega)} \\
    & =
    \langle v_0, w\rangle_{L^2(\Omega)} + \sum_{i=1}^m \langle v_i, \partial_i^* w\rangle_{L^2(\Omega)}\\
    & \le 
    \|v_0\|_{L^2(\Omega)} \|w\|_{L^2(\Omega)} + \sum_{i=1}^m \|v_i\|_{L^2(\Omega)} \|\partial_i^*w\|_{L^2(\Omega)}\\
    & \le
    \|(v_0, v_1, \ldots, v_m)\|_{L^2(\Omega)} \|w\|_{1, V^*},
\end{align}
where we applied the Cauchy-Schwarz inequality in the fourth line for the \(L^2\) inner product and in the fifth line for the Euclidean inner product. Now, we find for \(w \neq 0\) that
\begin{equation}
    \frac{\langle u, w\rangle_{L^2(\Omega)}}{\|w\|_{1, V^*}} \leq \|(v_0, v_1, \ldots, v_m)\|_{L^2(\Omega)}.
\end{equation}
The left hand side is now independent of the specific choice of the decomposition \(v_0, v_1, \ldots, v_m\) and the right hand side of \(w \in V\). This allows to simultaneously take the supremum over all \(w \neq 0\) on the left hand side and the infimum over all decompositions \(v_0, v_1,\ldots, v_m\) on the right hand side. Hence, yielding 
\(|u|_{-1, V^*} \leq |u|_{\inf}\).
Conversely, choose \(v_* \in V\) such that
\begin{equation}
    u = v_* + \sum_{i=1}^m \partial_i \partial_i^* v_*,
\end{equation}
given explicitly by \(v_* = M u\). Thus, we estimate
\begin{align}
    |u|_{\inf} & \leq \|(v_*, \partial_1^* v_*, \ldots, \partial_m^* v_*)\|_{L^2(\Omega)} = \|v_*\|_{1, V^*} = |u|_{-1, V^*},
\end{align}
where we applied Proposition~\ref{prop:laplacian} in the last step.
\hfill $\square$

\newpage

%Bibliography
\bibliographystyle{plain}
\bibliography{references}  

@book{
    Lions1972,
    author    = {J. L. Lions and E. Magenes},
    title     = {Non-Homogeneous Boundary Value Problems and Applications},
    volume    = {1},
    series    = {Grundlehren der mathematischen Wissenschaften},
    publisher = {Springer-Verlag},
    address   = {Berlin, Heidelberg},
    year      = {1972},
    doi       = {10.1007/978-3-642-65161-8},
    isbn      = {978-3-642-65163-2},
}

@book{
    Brezis2011, 
    title={Functional Analysis, Sobolev Spaces and Partial Differential Equations}, 
    doi={10.1007/978-0-387-70914-7}, 
    publisher={Springer New York}, 
    author={Brezis, Haim}, 
    year={2011}, 
    language={en} 
}

@book{
    Jost2007,
    author    = {Jost, J{\"u}rgen},
    title     = {Partial Differential Equations},
    series    = {Graduate Texts in Mathematics},
    edition   = {2},
    publisher = {Springer},
    address   = {New York, NY},
    year      = {2007},
    isbn      = {978-1-4419-2380-6},
    doi       = {10.1007/978-0-387-49319-0}
}

@book{
    Ern2004,
    author    = {Ern, Alexandre and Guermond, Jean-Luc},
    title     = {Theory and Practice of Finite Elements},
    series    = {Applied Mathematical Sciences},
    volume    = {159},
    edition   = {1},
    publisher = {Springer},
    address   = {New York, NY},
    year      = {2004},
    isbn      = {978-0-387-20574-8},
    doi       = {10.1007/978-1-4757-4355-5}
}

@book{
    LeVeque2007,
    title={Finite difference methods for ordinary and partial differential equations: steady-state and time-dependent problems},
    series={Other Titles in Applied Mathematics},
    author={LeVeque, Randall J},
    year={2007},
    publisher={SIAM},
    isbn={978-0-89871-629-0},
    doi={10.1137/1.9780898717839},
}

@article{
    Eymard2000,
    author={Eymard, Robert and Gallou{\"e}t, Thierry and Herbin, Rapha{\`e}le},
    title={Finite volume methods},
    booktitle={Handbook of Numerical Analysis},
    series={Handbook of numerical analysis},
    volume={7},
    pages={713--1018},
    publisher={Elsevier},
    year={2000},
    doi={10.1016/S1570-8659(00)07005-8},
}

@incollection{
    Bernardi1997,
    title = {Spectral methods},
    series = {Handbook of Numerical Analysis},
    publisher = {Elsevier},
    volume = {5},
    pages = {209--485},
    year = {1997},
    booktitle = {Techniques of Scientific Computing (Part 2)},
    issn = {1570-8659},
    doi = {10.1016/S1570-8659(97)80003-8},
    author = {Christine Bernardi and Yvon Maday}
}

@book{
    Canuto2007,
    title={Spectral Methods: Fundamentals in Single Domains},
    author={Canuto, Claudio and Hussaini, M Yousuff and Quarteroni, Alfio and Zang, Thomas A.},
    series    = {Scientific Computation},
    publisher = {Springer},
    address   = {Berlin, Heidelberg},
    year      = {2006},
    isbn      = {978-3-540-30725-9},
    doi       = {10.1007/978-3-540-30726-6},
}

@book{
    Li2004,
    author    = {Li, Shaofan and Liu, Wing Kam},
    title     = {Meshfree Particle Methods},
    edition   = {1},
    publisher = {Springer},
    address   = {Berlin, Heidelberg},
    year      = {2004},
    isbn      = {978-3-540-22256-9},
    doi       = {10.1007/978-3-540-71471-2}
}

@article{
    Bramble1970,
    author = {Bramble, James H. and Schatz, Alfred H.},
    title = {Rayleigh-Ritz-Galerkin methods for dirichlet's problem using subspaces without boundary conditions},
    journal = {Communications on Pure and Applied Mathematics},
    volume = {23},
    number = {4},
    pages = {653-675},
    doi = {10.1002/cpa.3160230408},
    year = {1970}
}

@article{
    Babuska1973,
    author = {Ivo Babu{\v{s}}ka},
    title = {The Finite Element Method with Penalty},
    journal = {Mathematics of Computation},
    volume = {27},
    number = {122},
    pages = {221--228},
    year = {1973},
    publisher = {American Mathematical Society},
    doi = {10.1090/S0025-5718-1973-0351118-5},
}

@article{
    Barrett1986,
    author={Barrett, John W. and Elliott, Charles M.},
    title={Finite element approximation of the Dirichlet problem using the boundary penalty method},
    journal={Numerische Mathematik},
    year={1986},
    month={Jul},
    day={01},
    volume={49},
    number={4},
    pages={343-366},
    issn={0945-3245},
    doi={10.1007/BF01389536},
}

@article{
    Raissi2019,
    author  = {M. Raissi and P. Perdikaris and G.E. Karniadakis},
    title   = {Physics-informed neural networks: A deep learning framework for solving forward and inverse problems involving nonlinear partial differential equations},
    journal = {Journal of Computational Physics},
    volume  = {378},
    pages   = {686--707},
    year    = {2019},
    doi     = {10.1016/j.jcp.2018.10.045},
}

@misc{
    Raissi2017a,
    title={Physics Informed Deep Learning (Part I): Data-driven Solutions of Nonlinear Partial Differential Equations}, 
    author={Maziar Raissi and Paris Perdikaris and George Em Karniadakis},
    year={2017},
    eprint={1711.10561},
    archivePrefix={arXiv},
    primaryClass={cs.AI},
    url={https://arxiv.org/abs/1711.10561}, 
}

@misc{
    Raissi2017b,
    title={Physics Informed Deep Learning (Part II): Data-driven Discovery of Nonlinear Partial Differential Equations}, 
    author={Maziar Raissi and Paris Perdikaris and George Em Karniadakis},
    year={2017},
    eprint={1711.10566},
    archivePrefix={arXiv},
    primaryClass={cs.AI},
    url={https://arxiv.org/abs/1711.10566}, 
}

@article{
    Lawal2022,
    author = {Lawal, Zaharaddeen Karami and Yassin, Hayati and Lai, Daphne Teck Ching and Che Idris, Azam},
    title = {Physics-Informed Neural Network (PINN) Evolution and Beyond: A Systematic Literature Review and Bibliometric Analysis},
    journal = {Big Data and Cognitive Computing},
    volume = {6},
    year = {2022},
    number = {4},
    issn = {2504-2289},
    doi = {10.3390/bdcc6040140}
}

@article{
    Cai2021,
    author    = {Cai, S. and Mao, Z. and Wang, Z. et al.},
    title     = {Physics-informed neural networks (PINNs) for fluid mechanics: a review},
    journal   = {Acta Mechanica Sinica},
    volume    = {37},
    pages     = {1727--1738},
    year      = {2021},
    doi       = {10.1007/s10409-021-01148-1}
}

@article{
    Toscano2025,
    author={Toscano, Juan Diego
    and Oommen, Vivek
    and Varghese, Alan John
    and Zou, Zongren
    and Ahmadi Daryakenari, Nazanin
    and Wu, Chenxi
    and Karniadakis, George Em},
    title={From PINNs to PIKANs: recent advances in physics-informed machine learning},
    journal={Machine Learning for Computational Science and Engineering},
    year={2025},
    month={Mar},
    day={11},
    volume={1},
    number={1},
    pages={15},
    issn={3005-1436},
    doi={10.1007/s44379-025-00015-1},
}

@article{
    Hu2024,
    title = {Tackling the curse of dimensionality with physics-informed neural networks},
    journal = {Neural Networks},
    volume = {176},
    pages = {106369},
    year = {2024},
    issn = {0893-6080},
    doi = {10.1016/j.neunet.2024.106369},
    author = {Zheyuan Hu and Khemraj Shukla and George Em Karniadakis and Kenji Kawaguchi},
}

@article{
    Wang2021,
    author = {Wang, Sifan and Teng, Yujun and Perdikaris, Paris},
    title = {Understanding and Mitigating Gradient Flow Pathologies in Physics-Informed Neural Networks},
    journal = {SIAM Journal on Scientific Computing},
    volume = {43},
    number = {5},
    pages = {A3055-A3081},
    year = {2021},
    doi = {10.1137/20M1318043},
}

@article{
    Maddu2022,
    doi = {10.1088/2632-2153/ac3712},
    year = {2022},
    month = {feb},
    publisher = {IOP Publishing},
    volume = {3},
    number = {1},
    pages = {015026},
    author = {Maddu, Suryanarayana and Sturm, Dominik and Müller, Christian L and Sbalzarini, Ivo F},
    title = {Inverse Dirichlet weighting enables reliable training of physics informed neural networks},
    journal = {Machine Learning: Science and Technology},
}

@article{
    Berrone2023,
    author  = {Berrone, Stefano and Canuto, Claudio and Pintore, Moreno and Sukumar, N.},
    title   = {Enforcing Dirichlet boundary conditions in physics-informed neural networks and variational physics-informed neural networks},
    journal = {Heliyon},
    volume  = {9},
    number  = {8},
    pages   = {e18820},
    year    = {2023},
    doi     = {10.1016/j.heliyon.2023.e18820}
}

@article{
    Suarez2024,
    doi = {10.1088/2632-2153/ad62ac},
    year = {2024},
    month = {jul},
    publisher = {IOP Publishing},
    volume = {5},
    number = {3},
    pages = {035029},
    author = {Suarez Cardona, Juan-Esteban and Hofmann, Phil-Alexander and Hecht, Michael},
    title = {Negative Order Sobolev Cubatures: Preconditioners of Partial Differential Equation Learning Tasks Circumventing Numerical Stiffness},
    journal = {Machine Learning: Science and Technology},
}

@book{
    Adams2003,
    author={Adams, R.A. and Fournier, J.J.F.},
    title={Sobolev Spaces},
    series={Pure and Applied Mathematics},
    publisher={Academic Press},
    address={Amsterdam},
    year={2003},
    isbn={978-0-12-044143-3},
}

@article{Gagliardo1958,
  author  = {Gagliardo, Emilio},
  title   = {Proprietà di alcune classi di funzioni in più variabili},
  journal = {Ricerche di Matematica},
  volume  = {7},
  year    = {1958},
  pages   = {102--137},
}

@article{Slobodeckij1958,
  author  = {Slobodeckij, L. N.},
  title   = {Generalized Sobolev spaces and their applications to boundary value problems of partial differential equations},
  journal = {Leningrad. Gos. Ped. Inst. U\v{c}ep. Zap.},
  volume  = {197},
  year    = {1958},
  pages   = {54--112},
}

@article{ 
    Gagliardo1957,
	author     = {Gagliardo, Emilio},
	title      = {Caratterizzazioni delle tracce sulla frontiera relative ad alcune classi di funzioni in \emph{n}-variabili},
	journal    = {Rendiconti del Seminario Matematico della Università di Padova},
	year       = {1957},
	volume     = {27},
	pages      = {284--305},
}

@techreport{
    Aronszajn1955,
    author      = {Aronszajn, Nachman},
    title       = {Boundary values of functions with finite Dirichlet integral},
    institution = {University of Kansas},
    type        = {Technical Report},
    number      = {14},
    year        = {1955},
    pages       = {77--94},
}

@article{
    Bochev97,
    author = {Bochev, Pavel},
    year = {1997},
    month = {02},
    title = {Experiences With Negative Norm Least-Square Methods For The Navier-Stokes Equations},
    volume = {6},
    journal = {Electronic Transactions on Numerical Analysis}
}

@article{
    Bramble97,
    ISSN = {00255718, 10886842},
    author = {James H. Bramble and Raytcho D. Lazarov and Joseph E. Pasciak},
    journal = {Mathematics of Computation},
    number = {219},
    pages = {935--955},
    publisher = {American Mathematical Society},
    title = {A Least-Squares Approach Based on a Discrete Minus One Inner Product for First Order Systems},
    volume = {66},
    year = {1997},
    doi = {10.1090/S0025-5718-97-00848-X},
}

@article{
    Kim02,
    author = {Sang Dong Kim and Byeong Chun Shin},
    title = {$H^{-1}$ least-squares method for the velocity–pressure–stress formulation of Stokes equations},
    journal = {Applied Numerical Mathematics},
    volume = {40},
    number = {4},
    pages = {451-465},
    year = {2002},
    doi = {10.1016/S0168-9274(01)00095-2},
}

@article{
    Parima2008,
    doi={10.1016/j.na.2008.02.074},
    title={Potential theory and applications in a constructive method for finding critical points of Ginzburg-Landau type equations},
    author={Kazemi, Parimah and Neuberger, J.W.},
    journal={Nonlinear Analysis: Theory, Methods \& Applications vol. 69 iss.3},
    year={2008},
    month={aug},
    volume={69},
    issue={3},
    page={925--930},
}

@book{
    Neuberger97,
    author = {J. W. Neuberger},
    title = {Sobolev Gradients and Differential Equations, in: Lecture Notes in Mathetmatics},
    volume = {1670},
    publisher = {Springer},
    address   = {Berlin, Heidelberg},
    year = {1997},
    doi = {10.1007/BFb0092831},
    isbn = {978-3-540-69594-3},
}

@article{
    Beurling59,
    author = {A. Beurling  and J. Deny },
    title = {Dirichlet Spaces},
    journal = {Proceedings of the National Academy of Sciences},
    volume = {45},
    number = {2},
    pages = {208-215},
    year = {1959},
    doi = {10.1073/pnas.45.2.208},
}

@article{
    Aleksov2016,
    title = {On the Markov inequality in the L2-norm with the Gegenbauer weight},
    journal = {Journal of Approximation Theory},
    volume = {208},
    pages = {9-20},
    year = {2016},
    issn = {0021-9045},
    doi = {10.1016/j.jat.2016.03.005},
    author = {D. Aleksov and G. Nikolov and A. Shadrin},
}

@article{
    Cohen2017,
    author  = {Cohen, Albert and Migliorati, Giovanni},
    title   = {Optimal weighted least-squares methods},
    journal = {The SMAI Journal of Computational Mathematics},
    year    = {2017},
    volume  = {3},
    pages   = {181--203},
    doi     = {10.5802/smai-jcm.24},
}

@article{
    Hecht2026,
    author = {Hecht, Michael and Hofmann, Phil-Alexander and Wicaksono, Damar and Hernandez Acosta, Uwe and Gonciarz, Krzysztof and Kissinger, Jannik and Sivkin, Vladimir and Sbalzarini, Ivo F},
    title = {Multivariate Newton interpolation in downward closed spaces reaches the optimal Bernstein–Walsh approximation rate},
    journal = {IMA Journal of Numerical Analysis},
    pages = {draf137},
    year = {2026},
    month = {03},
    doi = {10.1093/imanum/draf137},
}

@book{
    Quarteroni2007,
    author    = {Alfio Quarteroni and Riccardo Sacco and Fausto Saleri},
    title     = {Numerical Mathematics},
    edition   = {2},
    series    = {Texts in Applied Mathematics},
    publisher = {Springer},
    address   = {Berlin, Heidelberg},
    year      = {2007},
    doi       = {10.1007/b98885},
    isbn      = {978-3-540-34658-6}
}

@article{
    Virtanen2020,
    author  = {Virtanen, P. and Gommers, R. and Oliphant, T.E. and others},
    title   = {{SciPy} 1.0: Fundamental Algorithms for Scientific Computing in Python},
    journal = {Nature Methods},
    volume  = {17},
    pages   = {261--272},
    year    = {2020},
    doi     = {10.1038/s41592-019-0686-2}
}

@article{
    Liu1989,
    author  = {Liu, Dong C. and Nocedal, Jorge},
    title   = {On the Limited Memory BFGS Method for Large Scale Optimization},
    journal = {Mathematical Programming},
    volume  = {45},
    pages   = {503--528},
    year    = {1989},
    doi     = {10.1007/BF01589116}
}
\end{document}